\documentclass{article}
\usepackage{amsmath,amsthm,amssymb,amsfonts}
\usepackage{graphicx} 
\usepackage{enumerate}
\usepackage{blkarray}
\usepackage{subfigure}
\usepackage{array}
\usepackage{stackengine}
\usepackage{faktor}
\usepackage{fancyhdr}
\usepackage{graphicx}
\usepackage[svgnames]{xcolor}
\usepackage[labelfont=bf,width=0.8\textwidth,justification=centering]{caption}
\usepackage{tikz}
\usepackage{tikz-cd}
\usepackage{subfigure}
\usepackage[colorlinks=true,linkcolor=Indigo,citecolor=Indigo]{hyperref
 }
\usepackage{color}

\def\i{\bar{i}}
\def\v{\mathbf{v}}

\def\p{\mathbf{p}}
\def\M{\overline{M}}
\def\q{\mathbf{q}}
\newtheorem{theorem}{Theorem}[section]
\newtheorem{proposition}[theorem]{Proposition}
\newtheorem{definition}[theorem]{Definition}

\newtheorem{corollary}[theorem]{Corollary}
\newtheorem{lemma}[theorem]{Lemma}
\theoremstyle{definition}
\newtheorem{example}[theorem]{Example}
\newtheorem{observation}[theorem]{Observation}
\newtheorem{remark}[theorem]{Remark}
\title{Adaptive dynamics of direct reciprocity with $N$~rounds of memory }
\author{N. Balabanova, H. Duong, C. Hilbe }

\numberwithin{equation}{section}
\numberwithin{figure}{section}
  
\begin{document}
\maketitle

\noindent
{\bf Abstract:} The theory of direct reciprocity explores how individuals cooperate when they interact repeatedly. 
In repeated interactions, individuals can condition their behaviour on what happened earlier. 
One prominent example of a conditional strategy is Tit-for-Tat, which prescribes to cooperate if and only if the co-player did so in the previous round. 
The evolutionary dynamics among such memory-1 strategies have been explored in quite some detail. 
However, obtaining analytical results on the dynamics of higher memory strategies becomes increasingly difficult, due to the rapidly growing size of the strategy space.
Here, we derive such results for the adaptive dynamics in the donation game. 
In particular, we prove that for every orbit forward in time, there is an associated orbit backward in time that also solves the differential equation. 
Moreover, we analyse the dynamics by separating payoffs into a symmetric and an anti-symmetric part and demonstrate some properties of the anti-symmetric part. 
These results highlight some interesting symmetries that arise when interchanging player one with player two, and cooperation with defection.

\section{Introduction}
The prisoner's dilemma is a classic thought experiment in game theory that helps to analyse social phenomena such as altruism and cooperation~\cite{trivers1971evolution,Glynatsi:HSSC:2021}.  
This game describes a situation in which two individuals independently decide whether to cooperate or defect with one another. 
When the game is only played once, mutual defection is the only equilibrium -- see \cite{Rossetti:ETH:2023} for details. However, when the game is repeated sufficiently often, many additional equilibria appear, including ones that allow for full cooperation~\cite{friedman:RES:1971}. 

To explain how these cooperative behaviours might emerge in nature, researchers use the tools of evolutionary game theory~\cite{sigmund2010calculus,broom:book:2013}.
This literature assumes the players' strategies are adaptable.
Players can change their strategies over time depending on their relative success. 
The resulting dynamics can be described either with differential equations (in infinite populations), or stochastic processes (in finite populations), see~\cite{traulsen:PRL:2005}.
However, in either case, it becomes increasingly difficult to derive analytical results when the game allows for many strategies (as in the repeated prisoner's dilemma). 
To circumvent these problems, researchers often explore the dynamics numerically (as in \cite{nowak:PNAS:1993, brauchli:JTB:1999, eriksson:proceedings:2001, szabo:PRE:2005, garcia:PLoSOne:2012, stewart2013extortion, szolnoki:scirep:2014,McAvoy:PNASNexus:2022} and many others). 

In this paper, we take a purely mathematical approach by studying the game's adaptive dynamics \cite{metzadaptive, brannstrom2013hitchhiker}.
This approach has been previously used to analyse the repeated prisoner's dilemma among players with restricted memory~\cite{nowak:AAM:1990,imhof:PRSB:2010,allen:AmNat:2013,hilbe2013adaptive,Chen:NJP:2022,laporte2023adaptive}. 
A common assumption of these studies is that players only react to the co-player's very last move, or the last moves of both players. 
Herein, we aim to derive results for players with more memory. 
Specifically, we allow players to consider the last 2, 3, or more generally, $N$ rounds. 

We introduce our general framework in Section~\ref{Sec:Model}. 
In particular, we explain how repeated games among memory-$N$ players can be represented as a finite-state Markov chain~\cite{hauert1997effects,stewart:scirep:2016,Mamiya:PRE:2020,Ueda:RSOS:2021,Murase:PLoSCompBio:2023a}.
Using methods of linear algebra, in Section~\ref{sec: memory N repeated donation games} we investigate some properties of the correponding transition matrix. 
In particular, we describe its general structure and the natural transformations that this matrix lends itself to. 
We show that there exist only eight ``alternative perspectives'' for a two-player memory-$N$ game.
These eight perspectives arise by either changing the focal player, or by replacing cooperation with defection. 
We describe exactly how each of these transformations can be achieved through transforming the transition matrix. 

In Section~\ref{sec: adaptive dynamics} we explore the adaptive dynamics of the repeated Prisoner's Dilemma, generalising some results from \cite{hilbe2017memory} and \cite{laporte2023adaptive}. 
Additionally, we show that the dynamics can be represented as the sum of two (not uncoupled) vector fields, which we refer to as the symmetric and anti-symmetric parts. 
With the appropriate choice of the payoff vector, the original dynamics can be viewed as a perturbation of the anti-symmetric part; we describe this in detail in the case of memory 1. 
Section~\ref{sec: antisymm dynamics} is dedicated to investigating some  properties of the anti-symmetric adaptive dynamics, such as their equilibria and invariants.  
We demonstrate that a tit-for-tat strategy is always an equilibrium of the system. 
Moreover, we generalise a result demonstrated in \cite{laporte2023adaptive}, by showing that the anti-symmetric system always has $2^{N-1}$ conserved quantities.

Overall, our results highlight the rich structure inherent in the adaptive dynamics of repeated games. 
Although the dimension of the strategy space increases quickly in $N$, this structure can help us understand the evolution of direct reciprocity in more detail.

\section{Model} \label{Sec:Model}
 \subsection{Adaptive dynamics}

In the most classical setting, adaptive dynamics considers an asexually reproducing resident population characterised by some trait $x$.  
This population is regularly challenged by mutants with trait $y$. 
Whether or not the mutant can invade (and take over the population) depends on its invasion fitness~$A(y,x)$. 
This function describes the fitness advantage (or disadvantage) that the mutant faces in a homogeneous population of residents. 
Later, we will identify this invasion fitness with the payoff of a player with strategy $y$ against a resident with strategy $x$. 
Mutants successfully invade and subsequently become the majority if and only if (see \cite{hofbauer1998evolutionary})
\[
A(y,x)> A(x,x).
\]
This condition means: against residents,  mutants need to receive a higher payoff than the residents receive themselves. 

The invasion fitness is a way to quantify the evolutionary advantage of mutations. 
In many applications, this invasion function is well-behaved. 
Thus, when considering an evolving population as a whole, it makes sense to adopt the approach of 
\textit{meso-evolution} 
 \cite{metz2012adaptive}. At its core lie two main ideas: ({\it i})~evolution should be perceived through the lens of changing uniform traits in the population, and ({\it ii})~said changes in traits are presented as a continuous stream of mutations.  

Leveraging on these principles, one can formally write the equation for the direction of the mutation that yields the largest invasion fitness. 
 First proposed by Hofbauer and Sigmund in 1990 \cite{hofbauer1990adaptive}, the resulting adaptive dynamics takes the form
\begin{equation}
    \label{eq: adaptive dynamics}
    \dot{y} = \frac{\partial A(x,y)}{\partial y}\vert_{y=x}.
\end{equation}
These dynamics can be motivated by assuming that the  mutant trait~$y$ is infinitesimally close to the resident trait~$x$. 
At every subsequent time step, the mutant with the largest payoff advantage reaches fixation and forms the new resident population. 
This new population is then again challenged by those mutants that have the largest advantage in the new environment.

\subsection{Prisoner's dilemma}
In the following, we assume that the population members derive their fitness from interacting in a repeated prisoner's dilemma. 
The repeated prisoner's dilemma is a two-player game where in each round, players can choose between two actions. 
They can either \textit{cooperate} ($C$) or \textit{defect} ($D$) (see \cite{nowak1995invasion, beckenkamp2007cooperation}). 
Every round, depending on the choices of both of the players, they reap an appropriate payoff. 
Payoffs are often represented from the point of view of player 1 (who chooses a row), when interacting against player 2 (who chooses a column),
\begin{equation*}
\begin{split}
\begin{matrix}
\ \  \ C \ \ D \\
\begin{matrix}
C\\ D
\end{matrix}   \begin{pmatrix}
        R&S\\
        T&P
    \end{pmatrix}
    \end{matrix}
    \end{split}
    \end{equation*}
The standard assumptions are that the payoffs satisfy $T>R>P>S$ and $2R>T+S$.
On the one hand, these conditions ensure that mutual cooperation is preferred to mutual defection (because $R>P$). 
On the other hand, each individual is tempted to defect irrespective of the co-player's choice (because $T>R$ and $P>S$). 
When specified, we will additionally assume that our game satisfies the condition of {\it equal gains from switching} \cite{nowak:AAM:1990}. 
That is, we assume
\[
R+P=T+S.
\]
One commonly used instance of a prisoner's dilemma with equal gains from switching is the so-called donation game \cite{sigmund2010calculus}. 
Here, cooperation means to pay a cost $c$ for the co-player to get a benefit $b\!>\!c$. Defection means to pay no cost and to provide no benefit. 
Hence, $R\!=\!b\!-\!c$, $S\!=\!-c$, $T\!=\!b$, and $P\!=\!0$. 

As is commonly assumed in the literature, we suppose the game is repeated for infinitely many rounds, and that payoffs of future rounds are not discounted.

\subsection{Reactive strategies}
\label{sec: reactive strategy}

When interacting in a repeated prisoner's dilemma, each player needs to have a strategy. 
That is, each player needs a rule that determines how to play in each round, given the previous history of interactions. 
Perhaps the simplest non-trivial class of strategies is the set of reactive strategies~\cite{nowak:AAM:1990}. 
Here, a player's action in any given round only depends on what the co-player did in the previous round. 
We can represent reactive strategies as 2-tuples, $\p = (p_{1},p_{2})$ for the focal player~1 and $\q = (q_{1},q_{2})$ for the opponent player~2. 
The entry $p_{1}$ indicates the focal player's probability of cooperating given that in the previous round, their opponent cooperated. 
The second entry $p_{2}$ is the cooperation probability given that in the previous round, their opponent defected. The quantities $q_{1}$ and $q_{2}$ are defined analogously. 
Although reactive strategies are comparably simple, they exhibit a multitude of fascinating phenomena \cite{molnar2023reactive, wahl1999continuous, baklanov2021reactive}. 

When both players use reactive strategies, there is an explicit formula for their expected payoffs per round. 
To represent this formula, we use the notation by Hofbauer \& Sigmund \cite{hofbauer1998evolutionary}: let $r = p_{1}-p_{2}$, $r' = q_{1}-q_{2}$, $m = \frac{p_{2} + r q_{2}}{1 -r r'}$ and $m'= \frac{q_{2} + r p_{2}}{1 -r r'}$. 

Then player 1's payoff function $A$ is given by
\begin{equation} \label{eq:payff_reactive}
\begin{split}
A(\p,\q) &= R m m' + S m (1-m') + T(1-m')m + P(1-m)(1-m')\\&= (1-c)m m' -c m(1-m') + (1-m)m'\\
&=\frac{ b\left(\left(q_1-q_2\right) p_2 + q_2\right) - c \left(p_2 + \left(p_1-p_2\right) q_2\right)}{1 - \left(p_1-p_2\right)\left(q_1-q_2\right)}
\end{split}
\end{equation}
The explicit form of the resulting adaptive dynamics, as well as their properties can be found in \cite{hofbauer1998evolutionary}.

\subsection{Memory-$N$ strategies}

Instead of assuming that players only react to their co-player's last action, in the following we allow them to take into account both players' actions in the $N$ previous rounds. 
Thus, their decision in the $N+1$st round will be governed by all the choices in the $N$ previous turns. 
Again, we write the two players' memory-$N$ strategies as two vectors $\p$ and $\q$, respectively. 
When both players adopt memory-$N$ strategies, we sometimes speak of a memory-$N$ game. 

These vectors contain the players' conditional probability to cooperate, depending on the outcome of the last $N$ rounds. 
As a result, they lie in $\mathbb{R}^{2^{2N}}$ and have entries of the form $p_{i_1 i_2\ldots i_{2N}}$ where $i_j\!\in\!\{C,D\}$.  
 Each pair of indices $i_{2k-1}$ and $i_{2k}$ with $k\!\in\!\{1,\ldots,N\}$  describes the two players' choices in a given round. 
We use the convention that within each pair, the focal player's choice comes first. Moreover, the oldest remembered round is given by the first pair $(i_1,i_2)$, whereas the outcome of the most recent round is given by the last pair $(i_{2N-1},i_{2N})$. 

As an example, consider a memory-3 strategy's entry $p_{CCCDDC}$. 
This entry gives the probability that player~1 cooperates in the next round, given that three rounds ago, both players cooperated, two rounds ago player~1 cooperated and player~2 defected, and in the previous round, player~1 defected and player~2 cooperated.  

\begin{remark}
For brevity, 
   we will sometimes write $p_i$ for the element $p_{i_1\ldots i_{2N}}$ of the vector $\p$. These two types of notation will be used interchangeably. 
\end{remark}

\noindent
We assume the vectors $\p$ and $\q$ are written in lexicographic order with $C$ coming first,
\[
\p = \begin{pmatrix}
    p_{CC\ldots CC}\\
    p_{CC\ldots CD}\\
    p_{CC\ldots DC}\\
    p_{CC\ldots DD}\\
    \vdots\\
    p_{DD\ldots DD}\\
\end{pmatrix}, \ \ \q = \begin{pmatrix}
    q_{CC\ldots CC}\\
    q_{CC\ldots CD}\\
    q_{CC\ldots DC}\\
    q_{CC\ldots DD}\\
    \vdots\\
    q_{DD\ldots DD}\\
\end{pmatrix}
\]
Once the players' strategies $\p$ and $\q$ are given, we can describe the resulting game dynamics as a Markov chain. 
The states of this chain are the possible $2^{2N}$ outcomes of the past $N$ rounds.  
Let  $\mathbf{M}_N(\p,\q)$ denote the corresponding  transition matrix (we sometimes omit the variables and simply write $\mathbf{M}_N$). 
This matrix is stochastic. When all $\p, \ \q\in(0,1)^{2^{2N}}$, it has a unique left unit eigenvector $\nu$ corresponding to eigenvalue 1. 
That is, $\nu$ is the unique vector that satisfies $\nu^T\mathbf{M}_N =\nu^T$ and $\sum_i \nu_i = 1$. Sometimes, this vector is referred to as the Markov chain's \textit{invariant distribution}~\cite{norris1998markov}. We describe the precise form of $\mathbf{M}_N$ further below. 

The information about a player's (one-round) payoff after a certain outcome of $N$ choices can be encoded in a $2^{2N}$ dimensional \textit{payoff vector} $\mathbf{f}_N$.  
Then the payoff function for the repeated game is defined as the first player's weighted average payoff across all rounds, 
\begin{equation}
    \label{eq: payoff function initial}
    A(\p,\q) = \left<\nu(\p,\q),\mathbf{f}_N\right>.
\end{equation}
The corresponding adaptive dynamics are then given by
\[
\dot{\p} =\frac{\partial A(\p,\q)}{\partial\p}\Big|_{\p=\q}= \frac{\partial \left<\nu(\p,\q),\mathbf{f}_N\right> }{\partial\p}\Big|_{\p=\q}.
\]
Having an evolutionary advantage in this model corresponds to getting the largest expected profit. For a homogeneous resident population with strategy~$\q$, adaptive dynamics thus points towards the (infinitesimally close) mutant strategy $\p$ that performs best against the resident.

The simplest example of the payoff function and corresponding adaptive dynamics is in memory 1. In this case, by using the formalism of Press \& Dyson~\cite{press2012iterated}, one can write the payoff function as
\[
A(\p,\q)  :=\frac{\left| 
\begin{array}{cccc}
 p_{CC} q_{CC}-1 & p_{CC}-1 & q_{CC}-1 & f_1 \\
 p_{CD} q_{DC} & p_{CD}-1 & q_{DC} & f_2 \\
 p_{DC} q_{CD} & p_{DC} & q_{CD}-1 & f_3  \\
 p_{DD} q_{DD} & p_{DD} & q_{DD} & f_4 \\
\end{array}
\right|}{\left| 
\begin{array}{cccc}
 p_{CC} q_{CC}-1 & p_{CC}-1 & q_{CC}-1 & 1 \\
 p_{CD} q_{DC} & p_{CD}-1 & q_{DC} & 1 \\
 p_{DC} q_{CD} & p_{DC} & q_{CD}-1 & 1 \\
 p_{DD} q_{DD} & p_{DD} & q_{DD} & 1 \\
\end{array}
\right|},
\]
with $(f_1,f_2,f_3,f_4)\!=\!(R,S,T,P)$. 
In the special case of equal gains from switching, an explicit form of the adaptive dynamics can be found in \cite{laporte2023adaptive}; we provide the full general form in Appendix~\ref{sec: appendix}, equation (\ref{eq: adaptive dynamics general form}).

\subsection{Counting strategy}
\label{sec: counting strategy}
In the general case, the adaptive dynamics of a game among players with memory-1 strategies takes place on a four-dimensional space. However, there exists an invariant three-dimensional subspace within $\mathbb{R}^4$:
\begin{theorem}[LaPorte, Hilbe et al., \cite{laporte2023adaptive}]
    The condition $p_{CD}=p_{DC}$ is invariant on the trajectories of memory-1 adaptive dynamics. 
\end{theorem}
Memory-1 strategies with $p_{CD}=p_{DC}$ are referred to as \textit{counting strategies}. They only depend on {\it how many} of the two players cooperated in the previous round, while being independent of {\it who} cooperated. Following the notation in~\cite{laporte2023adaptive}, we describe the adaptive dynamics for counting strategies in terms of a 3-tuple $(q_2,q_1,q_0)$. The subscript here encodes the number of players who cooperated in the previous round: $p_{CC} = q_2, p_{CD}=p_{DC} = q_1$ and $p_{DD} = q_0$. We provide an explicit form of the equations for counting strategies  in Appendix~\ref{sec: appendix}, equation (\ref{eq: explicit form counting strategy}).

\subsection{Recursive construction of $\mathbf{M}_N$ and the payoff vector }
\label{sec: recursive constructions}
In the following, we describe the general form of the transition matrix for games among memory-$N$ players. 
In general, each entry of this transition matrix can be defined directly (see, for example, \cite{hauert1997effects}). However, in this section we point out that the transition matrix $\mathbf{M}_{N}$ for a game among memory-$N$ players can be constructed recursively from the transition matrix $\mathbf{M}_{N-1}$; additionally, we provide the algorithm for this construction. A similar recursive procedure exists for the payoff vector $\mathbf{f}_{N}$, which we will describe as well. These constructions will later be useful to derive some of the symmetries inherent in these transition matrices, see Section~\ref{sec: memory N repeated donation games}.

Recall that in the indices of elements of $\p$ and $\q$, the choice of the focal player comes first; we introduce some notation leveraging on this duality. 
\begin{definition}
Consider $p_{i_1\ldots i_{2N}}$ or $p_i$, an element of the vector $\mathbf{p}$. To each  $p_i$ we can associate one element of $\q$ that is telling the same "story" of the players' choices. We denote it by $q_{\i_1\ldots \i_{2N
}}$  or $q_{\i}$ for short.
\end{definition}

The connection between indices of $p_{i_1\ldots i_{2N}}$ and   $q_{\i_1\ldots \i_{2N}}$ is as follows: within these indices, consider pairs representing each game. Then any such pair of the form $CC$ or $DD $ in $i_1\ldots i_{2N}$ will remain unchanged. However, every $CD$ will become $DC$ and vice versa.  For example, in a memory 2 game, the element $p_{CDDD}$ will correspond to $q_{DCDD}$ and $p_{CDCD}$ to $q_{DCDC}$. Based on this notation, we can provide the following recursive construction to define the transition matrix $M_N$.

\begin{lemma}
\label{lem: construction matrix}
For $N\!\ge\!2$, let the $2^{2(N-1)}\times2^{2(N-1)}$ transition matrix for the memory-$(N\!-\!1)$ game have the form
\begin{equation}
    \label{eq: memory N-1 matrix}
  \mathbf{M}_{N-1}:=  \begin{pmatrix}
        M_1\\
        M_2\\
        M_3\\
        M_4
    \end{pmatrix}
\end{equation}
where $M_1, M_2, M_3, M_4$ are $2^{2(N-2)}\times2^{2(N-1)}$ matrices that are obtained by  cutting the matrix $\mathbf{M}_{N-1}$ horizontally into four submatrices of equal dimensions. 

Then the  transition  matrix for the  memory $N$ game has the following form:
\begin{equation}
    \label{eq: memory N matrix}
 \mathbf{M}_N  =   \begin{pmatrix}
        M'_1& 0& 0&0\\
        0&M'_2&0&0\\
        0&0&M'_3&0\\
        0&0&0&M'_4\\
        M''_1& 0& 0&0\\
        0&M''_2&0&0\\
        0&0&M''_3&0\\
        0&0&0&M''_4\\
        M'''_1& 0& 0&0\\
        0&M'''_2&0&0\\
        0&0&M'''_3&0\\
        0&0&0&M'''_4\\
        M''''_1& 0& 0&0\\
        0&M''''_2&0&0\\
        0&0&M''''_3&0\\
        0&0&0&M''''_4
    \end{pmatrix}.
\end{equation}
Here,  "0" denotes a $2^{2(N-2)}\times2^{2(N-1)}$  matrix consisting of zeros. The dashed matrices are constructed in the following way: to the beginning of the list of indices of each $p_{i_1\ldots i_{2(N-1)}}$ from any of the matrices $M_i$ we append   $CC$, $CD$, $DC$ or $DD$, to make it respectively into $M'_i$, $M''_i$, $M'''_i$ or $M'''_i$.
\end{lemma}
\begin{proof}
The statement is correct for $N=2$. Assume that it holds for some $N-1$. This means that the overall  structure of the matrix $\mathbf{M}_{N-1}$ will be  as in (\ref{eq: memory N matrix}), with four  "diagonals" made up of the following quadruples (they occupy four consecutive  entries in each row of  the transition matrix):

\begin{equation}
\label{eq: p quadruples}
p_{i}q_{\bar{i}} \ \ \ \ p_{i}(1-q_{\bar{i}})\ \ \ \ (1-p_{i})q_{\bar{i}}\ \ \ \ (1-p_{i})(1-q_{\bar{i}})
\end{equation}

Following all four diagonals down and to the right, one descends in lexicographic order of $p_{i_1\ldots i_{2N-2}}$ (but not in the lexicographic order of $q_{i_1\ldots i_{2N-2}}$!). 

   Let the probability distribution for the cooperation of player $\p$ at $k$th step be denoted by $\nu(k)$ -- we introduce the index to avoid confusion with the invariant distribution.  Then the following relation holds for $\nu(k+1)$:
    \[
    \nu(k)^T\mathbf{M}_N = \nu^T(k+1)
    \]
 Since in $\nu(k+1)$  the oldest games correspond to indices on the left and the memory $N$ is a fixed number,  with each new step in the game, the  two leftmost indices must be discarded, and two new indices must be appended on the right.
 
 These new entries, describing the choices made in the $k+1$st step in the game, are determined by the entries of $\nu(k)$ and the elements of the matrix $\mathbf{M}_N$. For example,
    \begin{equation}
    \begin{split}
          \nu_{i_1i_2 i_3i_4\ldots i_{2N-2} CC}(k+1) &= \nu_{CC i_1\ldots i_{2N-2}}(k)p_{CC i_1\ldots i_{2N-2}}q_{CC \i_1\ldots \i_{2N-2}} +\\& \nu_{CD i_1\ldots i_{2N-2}}(k)p_{CD i_1\ldots i_{2N-2}}q_{DC \i_1\ldots \i_{2N-2}} + \\&\nu_{DC i_1\ldots i_{2N-2}}(k)p_{DC i_1\ldots i_{2N-2}}q_{CD \i_1\ldots \i_{2N-2}} +\\&\nu_{DD i_1\ldots i_{2N-2}}(k)p_{DD i_1\ldots i_{2N-2}}q_{DD \i_1\ldots \i_{2N-2}}.
          \end{split}
       \end{equation}
       Analogous equalities hold for $\nu_{i_1i_2 i_3i_4\ldots i_{2N-2} CD}(k+1) $, $\nu_{i_1i_2 i_3i_4\ldots i_{2N-2} DC}(k+1) $ and $\nu_{i_1i_2 i_3i_4\ldots i_{2N-2} DD}(k+1) $: in those cases  we multiply the antries of $\nu(k)$ by  $(1-q_{\bar{i}})p_i$, $(1-p_i)q_{\bar{i}}$ or $(1-p_i)(1-q_{\bar{i}})$ respectively. 
       
       Therefore, the last two indices in any element of $\nu(k+1)$ are entirely  determined by what functions of $p_i$ and $q_{\bar{i}}$ we multiply by. Since $\nu(k+1)$ is arranged in lexicographic order, the four elements  $\nu_{i_1i_2 i_3i_4\ldots i_{2N-2} CC}$, $\nu_{i_1i_2 i_3i_4\ldots i_{2N-2} CD}(k+1) $, $\nu_{i_1i_2 i_3i_4\ldots i_{2N-2} DC}(k+1) $ and $\nu_{i_1i_2 i_3i_4\ldots i_{2N-2} DD}(k+1) $ must be together in this exact order: since we are multiplying them by the same entries of $\nu(k)$ this means that the rows of  $\mathbf{M}_N$ have  quadruples of the form (\ref{eq: p quadruples}).

       We go through the remaining indices $i_1,\ldots i_{2N-2}$ in lexicographic order; with each change we "descend" four coordinates in $\nu(k+1)$ (since the last two indices are determined by  the functions of $p_i$ and $q_{\bar{i}}$); therefore, the number of the row on which each consecutive quadruple lies must also increase by 1. Since the sum consists of four elements, we must have the same four-diagonal structure as in (\ref{eq: memory N matrix}) for $\mathbf{M}_N$, with quadruples (\ref{eq: p quadruples}) arranged in lexicographic order in $p$. This is precisely the structure that we obtain by the construction described in the statement of the Lemma.         
\end{proof}
In addition to formalising the form of the transition matrix, we want to describe a recursive way of constructing the payoff vector. 
\begin{lemma}
\label{st: recursive payoff vector}
    Suppose the payoff matrix has the form 
    \[
    \begin{pmatrix}
        R&S\\
        T&P
    \end{pmatrix}.
    \]
    Then the payoff vector $\mathbf{f}_N$ for the memory-$N$ game can be constructed recursively. Let 
    \begin{equation}
    \label{eq: recursive noq quite payoff}
    \begin{split}
       \widetilde{ \mathbf{f}}_1  = \begin{pmatrix} R\\S\\T\\P \end{pmatrix}, \ \ \widetilde{\mathbf{f}}_{N} = \begin{pmatrix} \widetilde{\mathbf{f}}_{N-1}+R\mathbf{1}_{N-1}\\\widetilde{\mathbf{f}}_{N-1}+S\mathbf{1}_{N-1}\\\widetilde{\mathbf{f}}_{N-1}+T\mathbf{1}_{N-1}\\\widetilde{\mathbf{f}}_{N-1}+P\mathbf{1}_{N-1}\end{pmatrix}, 
        \end{split}
    \end{equation}
    
    where $\mathbf{1}_N$ is a $2^{2N}$ dimensional vector whose entries are all equal to 1. Then 
    \begin{equation}
        \mathbf{f}_N =\frac{1}{N}\widetilde{\mathbf{f}_N}.
    \end{equation}
\end{lemma}
\begin{proof}
    The proof is purely computational: The payoff vector has to become four times longer, and the additional term comes from adding one more $C$-$D$ pair to the memory. Since we arrange entries in lexicographic order, we assume that this new pair is placed on the left - this is why we can add $R\mathbf{1}_{N-1}$,$S\mathbf{1}_{N-1}$, $T\mathbf{1}_{N-1}$, $P\mathbf{1}_{N-1}$ in this order. 

    The last step is to 'normalise' the payoff vector to make payoffs comparable to the one-shot (non-repeated) game. This is accomplished by dividing the recursively constructed vector $\widetilde{\mathbf{f}}_N$ by $N$. 
\end{proof}
\begin{remark}
\label{st: remark payoff recursive}
    In Lemma \ref{st: recursive payoff vector}, instead of computing $\widetilde{\mathbf{f}}_N$ recursively, we could also  use the following recursive definition that applies to $\mathbf{f}_N$ directly, 
    \begin{equation}
        \mathbf{f}_N = \begin{pmatrix}
            \frac{N-1}{N}\mathbf{f}_{N-1} + \frac{1}{N}R\mathbf{1}_{N-1}\\
            \frac{N-1}{N}\mathbf{f}_{N-1} + \frac{1}{N}S\mathbf{1}_{N-1}\\
            \frac{N-1}{N}\mathbf{f}_{N-1} + \frac{1}{N}T\mathbf{1}_{N-1}\\
            \frac{N-1}{N}\mathbf{f}_{N-1} + \frac{1}{N}P\mathbf{1}_{N-1}
        \end{pmatrix}.
    \end{equation}
\end{remark}

As a final preparation, we provide a more direct method to compute the payoff function $A(\p,\q)$. 
In our definition~\eqref{eq: payoff function initial}, we give a formula that depends on the Markov chain's invariant distribution $\nu(\p,\q)$. 
This invariant distribution is itself the solution of an implicit equation. 
The following result provides a more immediate payoff formula, based on Press \& Dyson's formalism \cite{press2012iterated}.

\begin{lemma}[\cite{hilbe2017memory}]
\label{st: payoff function}
Consider a transition matrix $\mathbf{M}_N$ and subtract the identity matrix $I$ of the appropriate dimension from it, then replace the last column of the obtained matrix by the vector $\mathbf{f}_N$. We denote such a  matrix by $\begin{pmatrix} \widetilde{\mathbf{M}}_N(\mathbf{p},\mathbf{q}) & \mathbf{f}_N \end{pmatrix}$. If $\mathbf{1}\in \mathbb{R}^{2^{2N}}$ is a vector with all entries equal to 1, then the payoff function $A(\p,\q)$ is given by 
\begin{equation}
    \label{eq:payoff function}
A(\mathbf{p},\mathbf{q}) = \frac{\begin{vmatrix}\widetilde{\mathbf{M}}_N(\mathbf{p},\mathbf{q}) & \mathbf{f}_N\end{vmatrix}}{\begin{vmatrix}\widetilde{\mathbf{M}}_N(\mathbf{p},\mathbf{q}) & \mathbf{1}_N\end{vmatrix}}, 
\end{equation}
where $|M|$ is the determinant of the matrix $M$.
\end{lemma}


\section{Memory N repeated donation game}
\label{sec: memory N repeated donation games}
After these preparations, we start by discussing some general properties of the transition matrix and the payoff vector. The respective results for the memory-$N$ donation game generalise statements from \cite{laporte2023adaptive} and \cite{hilbe2017memory}. They lay the foundation for Section~\ref{sec: adaptive dynamics}, where we study the game's adaptive dynamics.

\subsection{Exchanging $C$ and $D$}
\label{sec: exchanging c and d}

In the following, we are interested in exploring certain symmetries. 
These symmetries arise by interchanging the labels $C$ and $D$ (and later on, by interchanging players 1 and 2). To this end, it will be useful to consider certain transformations of the players' memory-$N$ strategies. 
To interchange $C$ and $D$, we introduce a transformation $\varphi:[0,1]^{2^{2n}}\rightarrow[0,1]^{2^{2n}}$, defined by
$$\varphi\big( \p \big) = \varphi\big( (p_{C\ldots C},\ldots,p_{D\ldots D})\big) = (1-p_{D\ldots D},\ldots,1\!-\!p_{C\ldots C}).$$
The vector $\varphi(\p)$ has the inverse interpretation of $\p$.
For example, the first entry of $\mathbf{p}$ is the player's {\it cooperation} probability, given that both players {\it cooperated} in all previous rounds. 
Instead, the first entry of $\varphi(\p)$ gives the player's {\it defection} probability, given both players {\it defected} in all previous rounds. 
A similar interpretation applies to all other entries. 
In each case, cooperation~($C$) needs to be replaced by defection~($D$) and vice versa. 

To derive a formal statement on the relationship between memory-$N$ strategies and their transformations, consider a skew diagonal $2^{2N}\times 2^{2N}$ matrix with all skew diagonal entries equal to 1. We denote it by $J^8_N$ or $J^8$, when the dimension is clear from the context. More specifically, 
\[
J_N^8:=\begin{pmatrix}
    0&\ldots&0&1\\
    0&\ldots&1&0\\
    \vdots&&&\vdots\\
    1&\ldots&0&0
\end{pmatrix}.
\]

Let $\tau$ be an order 2 transformation of matrices that sends a matrix to its 'transpose' with respect to its skew diagonal: instead of reflecting with respect to the diagonal containing elements $a_{ii}$, we are reflecting with respect to the one comprised out of $a_{i \ n+1-i}$.
Now we can show the main result of this subsection. 
\begin{theorem}
\label{thm: symmetry of M}
$ \mathbf{M}_N\big(\varphi(\p),\varphi(\q)\big) = \left(\mathbf{M}_N(\p,\q)^T\right)^{\tau} =J^8_N\cdot\mathbf{M}_N(\p,\q)\cdot J^8_N$
\end{theorem}
\begin{proof}

We prove the statement by induction on the length $N$ of the memory. 
The statement holds when $N=1$:
\begin{small}
\begin{equation*}
    \begin{split}
        \mathbf{M}_1(\p,\q) &= \begin{pmatrix}
            p_{CC} q_{CC} & p_{CC}(1- q_{CC}) & (1-p_{CC}) q_{CC} & (1-p_{CC})(1- q_{CC})  \\
            p_{CD} q_{DC} & p_{CD}(1- q_{DC}) & (1-p_{CD}) q_{DC} & (1-p_{DC})(1- q_{CD})  \\
            p_{DC} q_{CD} & p_{DC}(1- q_{CD}) & (1-p_{DC}) q_{CD} & (1-p_{DC})(1- q_{CD})  \\
            p_{DD} q_{DD} & p_{DD}(1- q_{DD}) & (1-p_{DD}) q_{DD} & (1-p_{DD})(1- q_{DD})  
        \end{pmatrix},\\
        \mathbf{M}_1\big(\varphi(\p),\varphi(\q)\big) &= \begin{pmatrix}
          (1-p_{DD})(1- q_{DD})&(1-p_{DD})q_{DD}&p_{DD}(1- q_{DD})& p_{DD} q_{DD}\\
          (1-p_{DC})(1- q_{CD}) & (1-p_{DC}) q_{CD}&p_{DC}(1- q_{CD})&p_{DC} q_{CD}\\
          (1-p_{DC})(1- q_{CD}) &(1-p_{CD}) q_{DC}&p_{CD}(1- q_{DC})& p_{CD} q_{DC}\\
          (1-p_{CC})(1- q_{CC}) &(1-p_{CC}) q_{CC}&p_{CC}(1- q_{CC})& p_{CC} q_{CC}
          \end{pmatrix}\\
          &= \left(\left(\mathbf{M}_1(p_{CC},\ldots q_{DD})\right)^T\right)^{\tau}
    \end{split}
\end{equation*}
\end{small}
Suppose now the claim holds for $N-1$; we set out to prove it for memory-$N$. To this end, we will make extensive use of the way we have recursively constructed transition matrices in Lemma \ref{lem: construction matrix}. For brevity, let us denote the respective matrices $M_i', M_i'', M_i'''$ and $M_i''''$ by $M_i^n$, where $n$ is the number of dashes. 

Consider $p_{i_1i_2\ldots i_{2N}}$ and its lexicographic reverse (we replace $C$ by $D$ and vice versa) $p_{j_1j_2\ldots j_{2N}}$. This operation  can be split up into two parts:
\begin{equation}
\label{eq: mapping}
    p_{i_1i_2i_3i_4\ldots i_{2N}}\mapsto p_{j_1j_2i_3i_4\ldots i_{2N}}\mapsto 1-p_{j_1j_2\ldots j_{2N}}. 
\end{equation}
Consider the first mapping. It reverses the first two indices; therefore, it is the pairwise exchange of the blocks of the matrix $\mathbf{M}_N$:
\begin{equation}
\label{eq: firstchange}
    \begin{pmatrix}
        M'_1& 0& 0&0\\
        0&M'_2&0&0\\
        0&0&M'_3&0\\
        0&0&0&M'_4\\
        M''_1& 0& 0&0\\
        0&M''_2&0&0\\
        0&0&M''_3&0\\
        0&0&0&M''_4\\
        M'''_1& 0& 0&0\\
        0&M'''_2&0&0\\
        0&0&M'''_3&0\\
        0&0&0&M'''_4\\
        M''''_1& 0& 0&0\\
        0&M''''_2&0&0\\
        0&0&M''''_3&0\\
        0&0&0&M''''_4
    \end{pmatrix}\to \begin{pmatrix}
        M''''_1& 0& 0&0\\
        0&M''''_2&0&0\\
        0&0&M''''_3&0\\
        0&0&0&M''''_4\\
        M'''_1& 0& 0&0\\
        0&M'''_2&0&0\\
        0&0&M'''_3&0\\
        0&0&0&M'''_4\\
        M''_1& 0& 0&0\\
        0&M''_2&0&0\\
        0&0&M''_3&0\\
        0&0&0&M''_4\\
        M'_1& 0& 0&0\\
        0&M'_2&0&0\\
        0&0&M'_3&0\\
        0&0&0&M'_4
    \end{pmatrix}
\end{equation}

The matrices $M_i^n$ themselves do not change. Now, the second operation from (\ref{eq: mapping}) is the lexicographic reverse for the memory $N-1$ game.  Let us examine  how it acts on  each of the blocks of the form
\begin{equation}
\label{eq:block}
\begin{pmatrix}
 M_1^n& 0& 0&0\\
        0&M^n_2&0&0\\
        0&0&M^n_3&0\\
        0&0&0&M^n_4
        \end{pmatrix}.
\end{equation}

As per our assumption, the statement of the theorem holds for the matrix $\mathbf{M}_{N-1}$. Therefore, the index exchange as in the  second mapping of (\ref{eq: mapping}) acts on the matrix $\mathbf{M}_{N-1}$ as the composition of two transpositions. 

One can observe that for  a square matrix $M$ of even dimension $(M^T)^{\tau} = (M^{\tau})^T$ and the geometric interpretation of this double transposition  is the combination of a vertical and a horizontal reflections of the matrix with respect to  the two orthogonal  axes drawn through its centre. We need the following key

\begin{observation}
    The pairwise  exchange of nonzero elements obtained from the combination of the horizontal and vertical reflections through the centre are the same for the matrix  (\ref{eq: memory N-1 matrix}) and the matrix (\ref{eq:block}). 
\end{observation}

Geometrically, that means that for all matrices $M^n_i$ all of their columns and rows are rewritten in the opposite order, and additionally, $M_1^n$ is exchanged with $M_4^n$, and $M_2^n$ with $M_3^n$. This happens for both (\ref{eq: memory N-1 matrix}) and (\ref{eq:block}), and the pairwise exchange of the elements within the matrices is identical in the two cases. 

Therefore, by the induction assumption,  after the change of variables, each of the blocks of the form (\ref{eq:block}) in the matrix (\ref{eq: memory N matrix}) undergoes two reflections: vertical and horizontal, with respect to the axes drawn through its centre. Together with the pairwise exchange of the blocks from the first mapping, this can be seen to be two reflections of the entire matrix $\mathbf{M}_N$, again with respect to the two orthogonal symmetry axes through the centre, which we know to be the same as the combination of the two transpositions. This proves our initial claim.

Lastly, it can be easily checked that the conjugation by $J^8$ is identical to the two transpositions, since multiplication by $J^8$ from the left reverses the order of the rows, and from the right the columns.
\end{proof}
The vector $J^8_N\mathbf{f}_N$ is connected to the vector $\mathbf{f}_N$ in the following way:
\begin{lemma}
\label{st: lemma payoff vector zero sum}
    For games with equal gains from switching, payoff vectors as constructed in Lemma \ref{st: recursive payoff vector} satisfy the following relationship,
    \begin{equation}
        -\mathbf{f}_N + K_N\mathbf{1}_N = J_N^8\mathbf{f}_N,
    \end{equation}
    where $K_N$ is a constant depending on $N$.
   
\end{lemma}

\begin{proof}
   Again, we proceed by induction. As demonstrated in \cite{laporte2023adaptive}, the statement is true for memory-1, since 
    \[
    \begin{pmatrix}
        -R\\
        -S\\
        -T\\
        -P\\
    \end{pmatrix}+ \begin{pmatrix} R+P\\S+T\\S+T\\R+P\end{pmatrix} = \begin{pmatrix}
        P\\T\\R\\S
    \end{pmatrix}. 
    \]
Here we have used the fact that $R+P=S+T=:K_1$. 

Assume that the statement is true for memory $N-1$. This implies existence of a constant  $K_{N-1}$, such that 
\[
-\mathbf{f}_{N-1}  + K_{N-1}\mathbf{1} = J_{N-1}\mathbf{f}_{N-1}
\]
 Lemma \ref{st: lemma payoff vector zero sum}  and Remark \ref{st: remark payoff recursive} describe the payoff vector for memory $N$; we are going to use the recursive construction described in the latter. 
Setting $K_N = \frac{N-1}{N}K_{N-1}+\frac{R+P}{N} $, we   introduce the following series of transformations for $\mathbf{f}_{N}$:
 \begin{small}
\begin{equation}
\label{eq: turning over profit}
\begin{split}
\mathbf{f}_{N} &=\begin{pmatrix}\frac{N-1}{N}\mathbf{f}_{N-1}+\frac{R}{N}\mathbf{1}_{N-1}\\\frac{N-1}{N}\mathbf{f}_{N-1}+\frac{S}{N}\mathbf{1}_{N-1}\\\frac{N-1}{N}\mathbf{f}_{N-1}+\frac{T}{N}\mathbf{1}_{N-1}\\\frac{N-1}{N}\mathbf{f}_{N-1}+\frac{P}{N}\mathbf{1}_{N-1}\end{pmatrix}\mapsto - \mathbf{f}_N = \begin{pmatrix}-\frac{N-1}{N}\mathbf{f}_{N-1}-\frac{R}{N}\mathbf{1}_{N-1}\\-\frac{N-1}{N}\mathbf{f}_{N-1}-\frac{S}{N}\mathbf{1}_{N-1}\\-\frac{N-1}{N}\mathbf{f}_{N-1}-\frac{T}{N}\mathbf{1}_{N-1}\\-\frac{N-1}{N}\mathbf{f}_{N-1}-\frac{P}{N}\mathbf{1}_{N-1} \end{pmatrix}\mapsto - \mathbf{f}_N + K_N\mathbf{1}_N \\ &= \begin{pmatrix}-\frac{N-1}{N}\mathbf{f}_{N-1}-\frac{R}{N}\mathbf{1}_{N-1}\\-\frac{N-1}{N}\mathbf{f}_{N-1}-\frac{S}{N}\mathbf{1}_{N-1}\\-\frac{N-1}{N}\mathbf{f}_{N-1}-\frac{T}{N}\mathbf{1}_{N-1}\\-\frac{N-1}{N}\mathbf{f}_{N-1}-\frac{P}{N}\mathbf{1}_{N-1} \end{pmatrix} + \begin{pmatrix}\frac{N-1}{N}K_{N-1}\mathbf{1}_{N-1} +\frac{R+P}{N}\mathbf{1}_{N-1}\\\frac{N-1}{N}K_{N-1}\mathbf{1}_{N-1} +\frac{S+T}{N}\mathbf{1}_{N-1}\\\frac{N-1}{N}K_{N-1}\mathbf{1}_{N-1} +\frac{S+T}{N}\mathbf{1}_{N-1}\\\frac{N-1}{N} K_{N-1}\mathbf{1}_{N-1} +\frac{R+P}{N}\mathbf{1}_{N-1}
    \end{pmatrix}\\&= \begin{pmatrix} \frac{N-1}{N}J_{N-1}^8\mathbf{f}_{N-1} + \frac{P}{N}\mathbf{1}_{N-1}\\\frac{N-1}{N}J_{N-1}^8\mathbf{f}_{N-1}+\frac{T}{N}\mathbf{1}_{N-1}\\\frac{N-1}{N}J_{N-1}^8\mathbf{f}_{N-1}+\frac{S}{N}\mathbf{1}_{N-1}\\\frac{N-1}{N}J_{N-1}^8\mathbf{f}_{N-1}+\frac{R}{N}\mathbf{1}_{N-1}\end{pmatrix}
     \end{split}
\end{equation}
\end{small}
The last expression in (\ref{eq: turning over profit}) is equal to $J_{N}^8\mathbf{f}_{N}$, which completes our proof. 
\end{proof}
\subsection{Exchanging $\p$ and $\q$}
\label{sec:exchanging p and q}
In this section, we decompose the payoff function $A(\p,\q)$ into two parts, $A_s(\p,\q)$ and $A_a(\p,\q)$. 
These two parts are respectively symmetric and anti-symmetric with respect to exchanging $\p$ and $\q$. Moreover, we demonstrate that the two expressions only differ in the associated payoff vector. This construction will become useful later in Section \ref{sec: adaptive dynamics}. 

\noindent
By Lemma \ref{st: payoff function}, the payoff function has the explicit form
\begin{equation*}
 A(\mathbf{p},\mathbf{q}) = \frac{\begin{vmatrix}\widetilde{\mathbf{M}}_N(\mathbf{p},\mathbf{q}) & \mathbf{f}_N\end{vmatrix}}{\begin{vmatrix}\widetilde{\mathbf{M}}_N(\mathbf{p},\mathbf{q}) & \mathbf{1}_N\end{vmatrix}}.
\end{equation*}
Also consider $J^2_N$, a recursively constructed matrix:
    $$
    J^2_1 := \begin{pmatrix}
        1&0&0&0\\
        0&0&1&0\\
        0&1&0&0\\
        0&0&0&1
    \end{pmatrix}
    $$
    Then for arbitrary $N\ge2$, let 
    \begin{equation}
    \label{eq: recursive J2}
    J^2_N:= \begin{pmatrix}
        J_{N-1}^2&0&0&0\\
        0&0& J_{N-1}^2&0\\
        0& J_{N-1}^2&0&0\\
        0&0&0& J_{N-1}^2
    \end{pmatrix}.
      \end{equation}
As with the matrix $J_N^8$, we will omit $N$ when the context permits and write $J^2$ instead.

\begin{theorem}
\label{st: decomposition memory N}
    For arbitrary memory $N$, the payoff function can be decomposed into a symmetric and a skew-symmetric part. Specifically, 
    \begin{equation*}
        A(\p,\q) = A_s(\p,\q) + A_a(\p,\q).
    \end{equation*}
   Here,
    \begin{equation}
A_{s}(\mathbf{p},\mathbf{q})
 = \frac12 \left(A(\p,\q) + A(\q,\p)\right) = \frac12\frac{\begin{vmatrix}\widetilde{\mathbf{M}}_N(\mathbf{p},\mathbf{q}) & J_N^2\mathbf{f}_N + \mathbf{f}_N\end{vmatrix}}{\begin{vmatrix}\widetilde{\mathbf{M}}_N(\mathbf{p},\mathbf{q}) & \mathbf{1}_N\end{vmatrix}},
\end{equation}
is a symmetric function in $\p$ and $\q$, that is $A_s(\mathbf{p},\mathbf{q})=A_s(\mathbf{q},\mathbf{p})$. Similarly, 
\begin{equation}
  A_{a}(\mathbf{p},\mathbf{q}) = \frac12\left(A(\p,\q) -A(\q,\p)\right) =  \frac12 \frac{\begin{vmatrix}\widetilde{\mathbf{M}}_N(\mathbf{p},\mathbf{q}) & \mathbf{f}_N-J_N^2\mathbf{f}_N \end{vmatrix}}{\begin{vmatrix}\widetilde{\mathbf{M}}_N(\mathbf{p},\mathbf{q}) & \mathbf{1}_N\end{vmatrix}},
\end{equation}
is an anti-symmetric function in $\p$ and $\q$, such that $A_a(\mathbf{p},\mathbf{q})=-A_a(\mathbf{q},\mathbf{q})$.  
\end{theorem}

It immediately follows from the linearity of the determinant that we only need to prove the statement for one of the parts. We choose $A_s(\p,\q)$. 
In order to proceed, we require two following lemmata. 
\begin{lemma}
To prove Theorem \ref{st: decomposition memory N}, it suffices to show that $J^2_N\cdot \mathbf{M}_N(\p,\q)\cdot J^2_N = \mathbf{M}_N(\q,\p)$.
\end{lemma}
\begin{proof}
  Assuming the statement above holds, i.e. 
\begin{equation}
\label{eq: exchange memory N}
    \mathbf{M}_N(\q,\p)  =J_N^2 \mathbf{M}_N(\p,\q) J_N^2,
    \end{equation}
  enables us to deduce that  
\[
\mathbf{M}_N(\q,\p) -I = J_N^2 \mathbf{M}_N(\p,\q) J_N^2 - I =  J_N^2 \left(\mathbf{M}_N(\p,\q) -I\right)J_N^2, \] 
since $(J_N^2)^2 = I$. 

If we exchange the last column of $M(\p,\q)$ for $\mathbf{f}_N$, then after conjugating by $J_N^2$ the last column will turn into $J_N^2\mathbf{f}_N$ (due to the form of $J_N^2$). Therefore, to obtain the statement of the theorem from (\ref{eq: exchange memory N}), we only need to append the appropriately modified last column to our matrix.

\end{proof}

\begin{lemma}
    $J_N^2 \cdot\mathbf{M}_N(\mathbf{p},\mathbf{q}) \cdot  J_N^2  =  \mathbf{M}_N(\mathbf{q},\mathbf{p})$.
\end{lemma}
\begin{proof}
A direct computation shows that the statement is true for memory-1. Assume that it holds for $N-1$; as per rules of induction, we set off to demonstrate it for memory $N$. 
By Lemma \ref{lem: construction matrix}, we know how to obtain the memory-$N$ transition matrix from the case of $N-1$. For convenience, we introduce some further notation. 
Consider the four matrices of the form
\begin{equation}
\label{eq: four matrices}
\begin{pmatrix}
    M_i^j\\
    0\\
    0\\
    0\\
\end{pmatrix}\ ,  \ \begin{pmatrix}
   0\\
     M_i^j\\
    0\\
    0\\
\end{pmatrix}\ , \ \begin{pmatrix}
   0\\
     0\\
    M_i^j\\
    0\\
\end{pmatrix}\ , \ \begin{pmatrix}
   0\\
     0\\
    0\\
    M_i^j\\
\end{pmatrix},
\end{equation}
that comprise the matrix $\mathbf{M}_N$. By construction, all matrices in (\ref{eq: four matrices}) are square and have the same dimension as $\mathbf{M}_{N-1}(\p,\q)$.

We will hence denote matrices from (\ref{eq: four matrices}) by $\overline{M}_i^j$, irrespectively of their dimension, since for our purposes it can be underestood to be $2^{2(N-1)}\times 2^{2(N-1)}$. Then  $\mathbf{M}_N$ can be  written as

\[
\mathbf{M}_N =\begin{pmatrix}
    \M_1'&\M_2'&\M_3'&\M_4'\\
    \M_1''&\M_2''&\M_3''&\M_4''\\
    \M_1'''&\M_2'''&\M_3'''&\M_4'''\\
    \M_1''''&\M_2''''&\M_3''''&\M_4''''\\
\end{pmatrix}.
\]
From the recursive construction of matrix $J_N^2$, see Eq.~\eqref{eq: recursive J2},  $J_N^2 \mathbf{M}
(\p,\q)J_N^2$ will have the form
\begin{small}
\begin{equation}
\label{eq: long matrix}
\begin{pmatrix}
   J_{N-1}^2 \M_1'J_{N-1}^2&J_{N-1}^2\M_3'J_{N-1}^2&J_{N-1}^2\M_2'J_{N-1}^2&J_{N-1}^2\M_4'J_{N-1}^2\\
    
   J_{N-1}^2 \M_1'''J_{N-1}^2&J_{N-1}^2\M_3'''J_{N-1}^2&J_{N-1}^2\M_2'''J_{N-1}^2&J_{N-1}^2\M_4'''J_{N-1}^2\\
J_{N-1}^2\M_1''J_{N-1}^2&J_{N-1}^2\M_3''J_{N-1}^2&J_{N-1}^2\M_2''J_{N-1}^2&J_{N-1}^2\M_4''J_{N-1}^2\\
J_{N-1}^2\M_1''''J_{N-1}^2&J_{N-1}^2\M_3''''J_{N-1}^2&J_{N-1}^2\M_2''''J_{N-1}^2&J_{N-1}^2\M_4''''J_{N-1}^2\\ 
\end{pmatrix}
\end{equation}
\end{small}
Note how the two middle 'columns' and 'rows' in (\ref{eq: long matrix}) exchanged places, as compared to the original form of $\mathbf{M}_N(\p,\q)$.
From our induction assumption and the recursive construction in Lemma~\ref{lem: construction matrix}, we conclude the following:
\[
J_{N-1}^2 \M_1^{i}(\p,\q)J_{N-1}^2 = J_{N-1}^2\begin{pmatrix} M_1^{i}(\p,\q)\\0\\0\\0\end{pmatrix}J_{N-1}^2 = \begin{pmatrix} M_1^{i}(\q,\p)\\0\\0\\0\end{pmatrix},
\]
since $J_{N-1}^2$ exchanges $\p$ and $\q$ for memory $N-1$ games. Analogously, 
\[
J_{N-1}^2 \M_4^{i}(\p,\q)J_{N-1}^2 = J_{N-1}^2\begin{pmatrix} 0\\0\\0\\M_4^{i}(\p,\q)\end{pmatrix}J_{N-1}^2 = \begin{pmatrix} 0\\0\\0\\M_4^{i}(\q,\p)\end{pmatrix}.
\]
Things are slightly more complicated for $\M_2^i$ and $\M_3^i$. But again, we observe that 
\[J_{N-1}^2 \M_2^{i}(\p,\q)J_{N-1}^2 = J_{N-1}^2\begin{pmatrix} 0\\M_2^{i}(\p,\q)\\0\\0\end{pmatrix}J_{N-1}^2 = \begin{pmatrix} 0\\0\\M_3^{i}(\q,\p)\\0\end{pmatrix}
\]
and vice versa, by the induction assumption (this follows from the two facts: firstly,  $J_{N-1}^2$ only  exchanges rows and columns; secondly, the indices of elements in $M_3$ in $M(\q,\p)$ are $q_{DC\ldots}p_{CD\ldots}$ which in matrix $M(\p,\q)$ lie in $M_2$).

This fully describes the action  of $J_{N-1}^2$ on each individual 'row' in (\ref{eq: long matrix}). However, the first two entries in the indices of $\p$ and $\q$ are opposite in the second and the third rows of $\mathbf{M}_N$. Therefore, in order to complete our exchange, we need to exchange these two 'rows'. The matrix in  (\ref{eq: long matrix}) satisfies that condition, which concludes our proof.  
\end{proof}

\noindent
Armed with this result and taking into consideration that $\vert J_N^2\vert=\pm1$ and $J^2_N\mathbf{1}_N = \mathbf{1}_N$, we can demonstrate that Theorem \ref{st: decomposition memory N} holds.
\begin{proof}[Proof of Theorem \ref{st: decomposition memory N}]
\begin{small}
\begin{equation}
    \begin{split}
\frac12(A(\p,\q) + A(\q,\p)) &=\frac12\left( \frac{\begin{vmatrix}\widetilde{\mathbf{M}}_N(\mathbf{p},\mathbf{q}) &
\mathbf{f}_N\end{vmatrix}}{\begin{vmatrix}\widetilde{\mathbf{M}}_N(\mathbf{p},\mathbf{q}) & \mathbf{1}_N\end{vmatrix}} + \frac{\begin{vmatrix}\widetilde{\mathbf{M}}_N(\mathbf{q},\mathbf{p}) &
\mathbf{f}_N\end{vmatrix}}{\begin{vmatrix}\widetilde{\mathbf{M}}_N(\mathbf{q},\mathbf{p}) & \mathbf{1}_N\end{vmatrix}}\right)\\
&=\frac12\left( \frac{\begin{vmatrix}\widetilde{\mathbf{M}}_N(\mathbf{p},\mathbf{q}) &
\mathbf{f}_N\end{vmatrix}}{\begin{vmatrix}\widetilde{\mathbf{M}}_N(\mathbf{p},\mathbf{q}) & \mathbf{1}_N\end{vmatrix}} + \frac{\vert J_N^2\vert\begin{vmatrix}\widetilde{\mathbf{M}}_N(\mathbf{q},\mathbf{p}) &
\mathbf{f}_N\end{vmatrix}\vert J_N^2\vert}{\vert J_N^2\vert\begin{vmatrix}\widetilde{\mathbf{M}}_N(\mathbf{q},\mathbf{p}) & \mathbf{1}_N\end{vmatrix}\vert J_N^2\vert}\right)\\
&=\frac12\left( \frac{\begin{vmatrix}\widetilde{\mathbf{M}}_N(\mathbf{p},\mathbf{q}) &
\mathbf{f}_N\end{vmatrix}}{\begin{vmatrix}\widetilde{\mathbf{M}}_N(\mathbf{p},\mathbf{q}) & \mathbf{1}_N\end{vmatrix}} + \frac{\left| J_N^2.\begin{pmatrix}\widetilde{\mathbf{M}}_N(\mathbf{q},\mathbf{p}) &
\mathbf{f}_N\end{pmatrix}. J_N^2\right|}{\left|J_N^2.\begin{pmatrix}\widetilde{\mathbf{M}}_N(\mathbf{q},\mathbf{p}) & \mathbf{1}_N\end{pmatrix}.J_N^2\right|}\right)\\
&=\frac12\left( \frac{\begin{vmatrix}\widetilde{\mathbf{M}}_N(\mathbf{p},\mathbf{q}) &
\mathbf{f}_N\end{vmatrix}}{\begin{vmatrix}\widetilde{\mathbf{M}}_N(\mathbf{p},\mathbf{q}) & \mathbf{1}_N\end{vmatrix}} + \frac{\begin{vmatrix}\widetilde{\mathbf{M}}_N(\mathbf{p},\mathbf{q}) &
J_N^2\mathbf{f}_N\end{vmatrix}}{\begin{vmatrix}\widetilde{\mathbf{M}}_N(\mathbf{p},\mathbf{q}) & \mathbf{1}_N\end{vmatrix}}\right)\\
&=\frac12\frac{\begin{vmatrix}\widetilde{\mathbf{M}}_N(\mathbf{p},\mathbf{q}) &
\mathbf{f}_N +J_N^2\mathbf{f}_N  \end{vmatrix}}{\begin{vmatrix}\widetilde{\mathbf{M}}_N(\mathbf{p},\mathbf{q}) & \mathbf{1}_N\end{vmatrix}}.
\end{split}
\end{equation}
\end{small}
The last two equalities stem from the linearity of the determinant.
This completes the proof of Theorem \ref{st: decomposition memory N}. 
\end{proof}

By construction, both $A_a(\p,\q)$ and $A_s(\p,\q)$ are payoff functions of appropriately defined games. These games have the same transition matrix $\mathbf{M}_N(\p,\q)$ as our original game, but with payoff vectors $\mathbf{f}_N - J^2_N\mathbf{f}_N$ and  $\mathbf{f}_N + J^2_N\mathbf{f}_N$, respectively.
$A_a(\p,\q)$ captures the difference in expected payoffs between the two strategies, whereas $A_s(\p,\q)$ returns the average payoff per player. Thus, adaptive dynamics for $A_a(\p,\q) $ can be expected to maximise the profit gap, whereas adaptive dynamics for $A_s(\p,\q)$ maximises the players' average payoffs.

\subsection{General symmetries }
Interestingly, both transformations described in Sections \ref{sec: exchanging c and d} and \ref{sec:exchanging p and q} lead to multiplications of the type $\mathbf{M}_N\mapsto O\mathbf{M}_NO^T$ for some orthogonal matrix $O$. 
This approach can be generalised to multiplying all the attributes of the game by an orthogonal matrix:
\[
\nu\mapsto O\nu, \ \mathbf{M}_N\mapsto O\mathbf{M}_NO^T, \mathbf{f}_N\mapsto O\mathbf{f}_N.
\]

\begin{lemma}
\label{st: lemma classification}
  The transformation
  \[
\nu\mapsto O\nu, \ \mathbf{M}_N\mapsto O\mathbf{M}_NO^T, \mathbf{f}\mapsto O\mathbf{f},
\]
with an orthogonal matrix $O$ leaves the payoff function invariant.
\end{lemma}

\begin{proof}
    The proof is pure computation. Since $\nu\mapsto O\nu$, $\nu^T\mapsto\nu^TO^T$  and 
    \[
    \nu^TO^TO\mathbf{M}_NO^T = \nu^TO^T =\nu O^T.
    \]
    Since we multiply $\mathbf{f}$ by the same matrix, the payoff function stays invariant. 
\end{proof}

\noindent
The above lemma naturally raises the question: which orthogonal matrices $O$ make the matrix $O\mathbf{M}_NO^T$ a memory-$N$ game transition matrix? To this end, we look for an orthogonal matrix $O$ that satisfies the two following conditions:
\begin{enumerate}
    \item The matrix $O\mathbf{M}_NO^T$ has the same four-diagonal structure as in Theorem \ref{lem: construction matrix}. 
    \item There must exist a change of variables that brings quadruples in every row of $O\mathbf{M}_NO^T$ to the form
    \begin{equation}
    \label{eq: quadruples}
    yz \ \ y(1-z)\ \ (1-y)z \ \ (1-y)(1-z)
    \end{equation}
    for some $y$ and $z$.
\end{enumerate}
These conditions might seem lax at first glance, but in fact they dramatically decrease the number of matrices $O$ that satisfy them. We call orthogonal matrices that satisfy (1) and (2)  \textit{admissible}. 

We can immediately narrow down the class of admissible matrices. 

\begin{lemma}
    Any admissible orthogonal matrix $O$ is a permutation matrix. 
\end{lemma}
\begin{proof}
    Suppose $O$ is admissible. Therefore, 
    \begin{equation}
        \label{eq: admissible}
        O\mathbf{M}_NO^T = \mathbf{M}_N'
    \end{equation}
    for a matrix $\mathbf{M}'$ that is also a transition matrix of some memory-$N$ game. Both $\mathbf{M}$ and $\mathbf{M}'$ are right stochastic and 
    \[
    \mathbf{1} = \mathbf{M}'\mathbf{1} = O\mathbf{M}O^T\mathbf{1}.
    \]
    Then 
    \[
    \mathbf{M}O^T\mathbf{1} = O^T\mathbf{1}.
    \]
    When all entries of $\p$ and $\q$ lie strictly within the $(0,1)^{2^{2N}}$ cube, the vector $\mathbf{1}$ is the only right eigenvector corresponding to the eigenvalue 1 \cite{norris1998markov}. Therefore, $O^T\mathbf{1}  = \mathbf{1}$ and $O$ is a left stochastic matrix. 

    On the other hand, everything that we have just concluded about $O$, holds for $O^T$, since we can rewrite the relation $O\mathbf{M}O^T =\mathbf{M}'$ as $O^T\mathbf{M}'O = \mathbf{M}$. Therefore, $O$ has to be a right stochastic matrix; all entries of $O$ are non-negative numbers. But orthogonal matrices that have only non-negative entries are permutations. 
\end{proof}
\begin{lemma}
\label{st: structure conservation memoty 1}
For memory-1 games, the only admissible matrices are
    \begin{equation}
        \label{eq: good J for memory 1}
        \begin{split}
 &  J^1 := \left(
\begin{array}{cccc}
 1 & 0 & 0 & 0 \\
 0 & 1 & 0 & 0 \\
 0 & 0 & 1 & 0 \\
 0 & 0 & 0 & 1 \\
\end{array}
\right), \ J^2:=\left(
\begin{array}{cccc}
 1 & 0 & 0 & 0 \\
 0 & 0 & 1 & 0 \\
 0 & 1 & 0 & 0 \\
 0 & 0 & 0 & 1 \\
\end{array}
\right), \ J^3:=\left(
\begin{array}{cccc}
 0 & 1 & 0 & 0 \\
 1 & 0 & 0 & 0 \\
 0 & 0 & 0 & 1 \\
 0 & 0 & 1 & 0 \\
\end{array}
\right), \\& J^4:=\left(
\begin{array}{cccc}
 0 & 1 & 0 & 0 \\
 0 & 0 & 0 & 1 \\
 1 & 0 & 0 & 0 \\
 0 & 0 & 1 & 0 \\
\end{array}
\right),\ J^5:=\left(
\begin{array}{cccc}
 0 & 0 & 1 & 0 \\
 1 & 0 & 0 & 0 \\
 0 & 0 & 0 & 1 \\
 0 & 1 & 0 & 0 \\
\end{array}
\right), \ J^6:=\left(
\begin{array}{cccc}
 0 & 0 & 1 & 0 \\
 0 & 0 & 0 & 1 \\
 1 & 0 & 0 & 0 \\
 0 & 1 & 0 & 0 \\
\end{array}
\right), \\& J^7:=\left(
\begin{array}{cccc}
 0 & 0 & 0 & 1 \\
 0 & 1 & 0 & 0 \\
 0 & 0 & 1 & 0 \\
 1 & 0 & 0 & 0 \\
\end{array}
\right), \ J^8:=
\left(
\begin{array}{cccc}
 0 & 0 & 0 & 1 \\
 0 & 0 & 1 & 0 \\
 0 & 1 & 0 & 0 \\
 1 & 0 & 0 & 0 \\
\end{array}
\right).
        \end{split}
    \end{equation}
\end{lemma}
\begin{proof}
    The proof is a matter of computation and a coordinate change.     
\end{proof}

\begin{remark}
It can be easily checked that the matrices $J^1\ldots J^8$ form a subgroup of the group of permutation matrices. Moreover, the following relations can be checked:
\begin{enumerate}
    \item $J^5 = J^4\cdot J^8$,
    \item $J^6 = J^3\cdot J^8$,
    \item $J^ 7= J^2\cdot J^8$.
\end{enumerate}
\end{remark}

\noindent
Transformations through conjugation by the above matrices can be interpreted as `admissible' alternative points of view on the game. Clearly, $J^1$ is the identity transformation. 
Moreover, we have already seen that the matrices $J^8$ and $J^2$ interchange the two actions ($C$ and $D$) and the two players (players~1 and~2), respectively.
Conjugation by the matrix $J^3$ is equivalent to the following coordinate change::
\begin{equation*}
\begin{aligned}[c]
    &p_{CC}' = p_{CD}\\
    &p_{CD}' = p_{CC}\\
    &p_{DC}'= p_{DD}\\
    &p_{DD}' = p_{DC}    \end{aligned}\ \ \ \ \ \ \ \ \ \ \ \ \ 
    \begin{aligned}[c]
&q_{CC}'= 1-q_{DC}\\
&q_{CD}' = 1-q_{CC}\\
&q_{DC}'= 1-q_{DD}\\
&q_{DD}' = 1-q_{CD}.
    \end{aligned}
\end{equation*}
Here, we exchange $C$ abd $D$ for the second player.
The matrix $J^4$ induces 
\begin{equation*}
\begin{aligned}[c]
    &p_{CC}' = 1-q_{DC}\\
    &p_{CD}' = 1-q_{DD}\\
    &p_{DC}'= 1-q_{CC}\\
    &p_{DD}' = 1-q_{CD}   \end{aligned}\ \ \ \ \ \ \ \ \ \ \ \ \ 
    \begin{aligned}[c]
&q_{CC}'= p_{CD}\\
&q_{CD}' = p_{DD}\\
&q_{DC}'= p_{CC}\\
&q_{DD}' = p_{DC}.
    \end{aligned}
\end{equation*}
That is, both the player labels and the action labels are interchanged.

Naturally, we aim to generalise Lemma \ref{st: structure conservation memoty 1} to arbitrary values of $N$. To do, so we need to introduce new objects. 
Consider now the set of recursively constructed matrices $J^1_N\ldots, J^8_N$. The procedure is  akin to that in Eq.~(\ref{eq: recursive J2}):  at the $N$th step, replace zero entries of $J_{N-1}^i$ with  $2^{2(N-1)}\times 2^{2(N-1)}$ zero matrices and nonzero entries with copies of $J^i_{N-1}$.  For example, the matrix $J_2^3$ has the following form:
\[
J_2^3 =
\left(\begin{array}{cccc|cccc|cccc|cccc}
0&0&0&0&0&1&0&0&0&0&0&0&0&0&0&0\\
0&0&0&0&1&0&0&0&0&0&0&0&0&0&0&0\\
0&0&0&0&0&0&0&1&0&0&0&0&0&0&0&0\\
0&0&0&0&0&0&1&0&0&0&0&0&0&0&0&0\\
\hline
0&1&0&0&0&0&0&0&0&0&0&0&0&0&0&0\\
1&0&0&0&0&0&0&0&0&0&0&0&0&0&0&0\\
0&0&0&1&0&0&0&0&0&0&0&0&0&0&0&0\\
0&0&1&0&0&0&0&0&0&0&0&0&0&0&0&0\\
\hline
0&0&0&0&0&0&0&0&0&0&0&0&0&1&0&0\\
0&0&0&0&0&0&0&0&0&0&0&0&1&0&0&0\\
0&0&0&0&0&0&0&0&0&0&0&0&0&0&0&1\\
0&0&0&0&0&0&0&0&0&0&0&0&0&0&1&0\\
\hline
0&0&0&0&0&0&0&0&0&1&0&0&0&0&0&0\\
0&0&0&0&0&0&0&0&1&0&0&0&0&0&0&0\\
0&0&0&0&0&0&0&0&0&0&0&1&0&0&0&0\\
0&0&0&0&0&0&0&0&0&0&1&0&0&0&0&0\\

\end{array}\right)
\]

\noindent
For the set $\{J_N^1,\ldots,J_N^8\}$, the following holds:
\begin{theorem}
\label{st: symmetries for memory N}
The matrices $J_N^i$ are the only admissible permutation matrices.
\end{theorem}

\begin{figure}
\centering
\tikzset{every picture/.style={line width=0.75pt}} 

\begin{tikzpicture}[x=1pt,y=1pt,yscale=-1,xscale=1]

\draw  [draw opacity=0] (256.83,78.17) -- (388.83,78.17) -- (388.83,211.17) -- (256.83,211.17) -- cycle ; \draw   (256.83,78.17) -- (256.83,211.17)(289.72,78.17) -- (289.72,211.17)(322.6,78.17) -- (322.6,211.17)(355.48,78.17) -- (355.48,211.17)(388.37,78.17) -- (388.37,211.17) ; \draw   (256.83,78.17) -- (388.83,78.17)(256.83,111.05) -- (388.83,111.05)(256.83,143.93) -- (388.83,143.93)(256.83,176.82) -- (388.83,176.82)(256.83,209.7) -- (388.83,209.7) ; \draw    ;

\draw (258.83,81.17) node [anchor=north west][inner sep=0.75pt]   [align=left] {$\displaystyle J_{_{N-1}}^{i_{1}}$ };
\draw (300.67,85.5) node [anchor=north west][inner sep=0.75pt]   [align=left] {$\displaystyle 0$};
\draw (333.17,85) node [anchor=north west][inner sep=0.75pt]   [align=left] {$\displaystyle 0$};
\draw (364.67,85) node [anchor=north west][inner sep=0.75pt]   [align=left] {$\displaystyle 0$};
\draw (267.17,119) node [anchor=north west][inner sep=0.75pt]   [align=left] {$\displaystyle 0$};
\draw (267.67,151) node [anchor=north west][inner sep=0.75pt]   [align=left] {$\displaystyle 0$};
\draw (268.17,184.5) node [anchor=north west][inner sep=0.75pt]   [align=left] {$\displaystyle 0$};
\draw (301.17,117.5) node [anchor=north west][inner sep=0.75pt]   [align=left] {$\displaystyle 0$};
\draw (300.67,183.5) node [anchor=north west][inner sep=0.75pt]   [align=left] {$\displaystyle 0$};
\draw (333.17,117.5) node [anchor=north west][inner sep=0.75pt]   [align=left] {$\displaystyle 0$};
\draw (333.17,151.5) node [anchor=north west][inner sep=0.75pt]   [align=left] {$\displaystyle 0$};
\draw (366.67,151.5) node [anchor=north west][inner sep=0.75pt]   [align=left] {$\displaystyle 0$};
\draw (366.17,184.5) node [anchor=north west][inner sep=0.75pt]   [align=left] {$\displaystyle 0$};
\draw (357.48,114.05) node [anchor=north west][inner sep=0.75pt]   [align=left] {$\displaystyle J_{_{N-1}}^{i_{2}}$ };
\draw (291.72,146.93) node [anchor=north west][inner sep=0.75pt]   [align=left] {$\displaystyle J_{_{N-1}}^{i_{3}}$ };
\draw (324.6,179.82) node [anchor=north west][inner sep=0.75pt]   [align=left] {$\displaystyle J_{_{N-1}}^{i_{4}}$ };

\end{tikzpicture}
\caption{A possible form of an admissible matrix $O$. The grid denotes $2^{2(N-1)}\times 2^{2(N-1)}$ submatrices that are either zero matrices or permutation matrices.}
\label{fig: matrix $O$}
\end{figure}
\begin{remark}
    These eight ``alternative points of view'' represent the exchange of $C$ and $D$ for one of the players, the change of the focal player, and combinations thereof.  We prove this theorem as a more general statement to avoid a case-by-case analysis that these exchanges are exactly what the matrices $J^1_N\ldots J^8_N$ do.
\end{remark}
\begin{proof}
    We split the proof of this statement into seven steps. 

\begin{enumerate}    
\item  We start by proving the following statement: 
\begin{lemma}
\label{st: O consists of blocks}
Any admissible matrix $O$ is `built' out of the $4\times4$ matrices from Lemma \ref{st: structure conservation memoty 1}, meaning that $O$ entirely consists out of $4\times4$ blocks of the given form.
\end{lemma}

\begin{proof} Suppose that the permutation matrix in question is denoted by $O$. Then the matrix $\mathbf{M}_N$ transforms into $O\mathbf{M}_NO^T$. In this setup, the left copy of $O$ is 'responsible' for some permutation $\sigma$ of the rows of $\mathbf{M}_N$; the same permutation is induced on the columns by right multiplication by $O^T$.

    Above we stated the two conditions that we require of $O\mathbf{M}_NO^T$; the second one -- (\ref{eq: quadruples}) -- is the exact type  of the quadruples that have to be the only nonzero elements in every  row. Observe that by our definition, these quadruples cannot be broken up by $
    \sigma$: if the four elements of (\ref{eq: quadruples}) are in columns $i, i+1,i+2,i+3$, then $\sigma(i), \sigma(i+1), \sigma(i+2), \sigma(i+3)$ put in some order   become  four consecutive numbers. 

    Moreover, not every permutation of (\ref{eq: quadruples}) is allowed: in fact, all the admissible permutations are given by the matrices in Lemma \ref{st: structure conservation memoty 1}.  

    Since we cannot separate the four entries  in (\ref{eq: quadruples}), $\sigma$ acts as a permutation on the set of quadruples. This concludes the proof of the lemma.
    \end{proof}
    
\item  The result is true for memory-2 strategies:
\begin{lemma}
\label{st: symmetries memory 2}
    Theorem \ref{st: symmetries for memory N} holds for $N=2$. 
\end{lemma}
We start with an example that provides the intuition for how the general proof works. 
\begin{example}
We already know from Lemma \ref{st: O consists of blocks} that the matrix $O$
is built of $4\times 4$ blocks that are matrices $J^i_1$. Therefore, for memory-2 it has to have a form similar to Figure \ref{fig: matrix $O$}, consisting of sixteen $4\times 4$ matrices, four of which are nonzero, with only one nonzero matrix in every 'row' and 'column'.

Assume that a nonzero matrix occupies the upper left quadrant in Figure~\ref{fig: matrix $O$}. For the purposes of this example, assume this matrix is $J^2_1$. Consider the expression $OMO^T$, in particular, the left multiplication by~$O$. For concreteness, we again write the matrix $O$ in a certain way. 
\begin{tiny}
\begin{equation}
\label{eq: n=2 matrices permutation}
\begin{split}
 &   \begin{pmatrix}
        J^2_1 &0&0&0\\
        0&0&0&J^{i_2}_1\\
        0&J^{i_3}_1&0&0\\
        0&0&J^{i_4}_1&0
    \end{pmatrix} \begin{pmatrix}
    \M_1'&\M_2'&\M_3'&\M_4'\\
    \M_1''&\M_2''&\M_3''&\M_4''\\
    \M_1'''&\M_2'''&\M_3'''&\M_4'''\\
    \M_1''''&\M_2''''&\M_3''''&\M_4''''\\
\end{pmatrix}
\begin{pmatrix}
        (J^2_1)^T &0&0&0\\
        0&0&(J^{i_3}_1)^T&0\\
        0&0&0&(J^{i_4}_1)^T\\
        0&(J^{i_2}_1)^T&0&0
    \end{pmatrix} \\
    &=\begin{pmatrix}
   J^2_1 \M_1'&J^2_1\M_2'&J^2_1\M_3'&J^2_1\M_4'\\
   J^{i_2}_1 \M_1''''&J^{i_2}_1\M_2''''&J^{i_2}_1\M_3''''&J^{i_2}_1\M_4''''\\
    J^{i_3}_1\M_1''&J^{i_3}_1\M_2''&J^{i_3}_1\M_3''&J^{i_3}_1\M_4''\\
J^{i_4}_1\M_1'''&J^{i_4}_1\M_2'''&J^{i_4}_1\M_3'''&J^{i_4}_1\M_4'''\\
   \end{pmatrix}\begin{pmatrix}
        (J^2_1)^T &0&0&0\\
        0&0&(J^{i_3}_1)^T&0\\
        0&0&0&(J^{i_4}_1)^T\\
        0&(J^{i_2}_1)^T&0&0
    \end{pmatrix} 
    \end{split}
\end{equation}
\end{tiny}
We are only interested in the first row of the penultimate matrix in (\ref{eq: n=2 matrices permutation}), namely, $\left(J^2_1 \M_1'\ J^2_1\M_2'\ J^2_1\M_3' \ J^2_1\M_4'\right)$. 
\begin{remark}
This form of (\ref{eq: n=2 matrices permutation}) also shows that we were free to assume that $J^2_1$ was located in the upper left corner of $O$: otherwise, instead of $\M_i'$ we would have $\M_i''$, $\M_i'''$ or $\M_i''''$ - this does not affect our reasoning
\end{remark}

As noted above, left multiplication by $O$ induces a permutation $\sigma$ on the rows, and right multiplication by $O^T$ a permutation $\sigma $ on the columns. In this case, to use the standerd permutation notation,  $\sigma  = (1)(23)(4)$.

We examine how it acts on the matrices $\M_i'$. Let
\begin{equation}
    \begin{split}
        \M_1' &= \begin{pmatrix}
            a_1&a_2&a_3&a_4\\
            0&0&0&0\\
             0&0&0&0\\
              0&0&0&0\\
        \end{pmatrix}, \ \M_2' = \begin{pmatrix}
             0&0&0&0\\
              b_1&b_2&b_3&b_4\\
             0&0&0&0\\
              0&0&0&0\\
        \end{pmatrix}\\
        \M_3' &= \begin{pmatrix}
             0&0&0&0\\
             0&0&0&0\\
             c_1&c_2&c_3&c_4\\
              0&0&0&0\\
        \end{pmatrix},\  \M_4' = \begin{pmatrix}
            0&0&0&0\\
             0&0&0&0\\
              0&0&0&0\\
               d_1&d_2&d_3&d_4\\
        \end{pmatrix}.\\
    \end{split}
\end{equation}
Then 
\begin{equation}
\begin{split}
      J^2_1  \M_1' &= \begin{pmatrix}
            a_1&a_2&a_3&a_4\\
            0&0&0&0\\
             0&0&0&0\\
              0&0&0&0\\
        \end{pmatrix},  \  J^2_1\M_2' = \begin{pmatrix}
             0&0&0&0\\
             0&0&0&0\\
             b_1&b_2&b_3&b_4\\
              0&0&0&0\\
        \end{pmatrix}\\
         J^2_1 \M_3' &= \begin{pmatrix}
             0&0&0&0\\
             c_1&c_2&c_3&c_4\\
             0&0&0&0\\
              0&0&0&0\\
        \end{pmatrix}, \   J^2_1\M_4' = \begin{pmatrix}
            0&0&0&0\\
             0&0&0&0\\
              0&0&0&0\\
               d_1&d_2&d_3&d_4\\
        \end{pmatrix}\\
    \end{split}
\end{equation}
So $J_1^2\M_1' = \M_1, J_1^2\M_4' = \M_4'$ and  $J_1^2\M_2' $ becomes the new $ \M_3'$ as well as $J_1^2\M_3' $ the new $ \M_2'$. Therefore, the first four rows of $O\mathbf{M}$ have the form  
\begin{equation}
\label{eq: first for rows for memory 2}
\setlength{\arraycolsep}{3pt}
\left(\begin{array}{cccc|cccc|cccc|cccc}
a_1&a_2&a_3&a_4&0&0&0&0&0&0&0&0&0&0&0&0\\
0&0&0&0&0&0&0&0&c_1&c_2&c_3&c_4&0&0&0&0\\
    0&0&0&0& b_1&b_2&b_3&b_4& 0&0&0&0&0&0&0&0\\
 0&0&0&0&0&0&0&0&0&0&0&0&d_1&d_2&d_3&d_4\\
\end{array}\right)
\end{equation}
Observe that the placement of the nonzero blocks in $O^T$ determines how the quadruples of columns are permuted; therefore, in order to force  the submatrix in (\ref{eq: first for rows for memory 2}) to the required structure, the quadruples of the columns must be permuted with the same permutation $\sigma$ as $J^2_1$ defines. 

Therefore, the location of the nonzero blocks in the matrix $O$ must be the same as the  location of the nonzero elements in $J^2_1$, namely, 
\[
O =  \begin{pmatrix}
        J^2_1 &0&0&0\\
        0&0&J^{i_2}_1&0\\
        0&J^{i_3}_1&0&0\\
        0&0&0&J^{i_4}_1
    \end{pmatrix}.
\]

As for the matrices $J^{i_j}_1$, they must induce the same permutation on the rows that they are multiplied by, as $J^2_1$, in order for the four-diagonal structure to persist. Hence, $J^{i_2}_1 =J^{i_3}_1=J^{i_4}_1=J^2_1 $ and $O= J^2_2$.
\end{example}

Armed with the intuition provided by the example above, we can move on to the general case. 

\begin{proof}[Proof of Lemma~\ref{st: symmetries memory 2}]
In Figure \ref{fig: matrix $O$}, let the only nonzero matrix of the first row be $J^i$. Consider multiplication by $O$ from the left only; the matrix $J^i$ is going to permute  the submatrices $\overline{M}_i^j$  -- the multiplication changes which row of each matrix is nonzero, as in the example above.

Therefore, the permutation of quadruples of columns that is done by the matrix $O^T$ on the left must exchange the matrices $\overline{M}_i^j$ in such a way that they reassemble into the four-diagonal structure; hence, the permutation of columns must coincide with the 
permutation of rows, and the nonzero block structure of $O$ must be identical to that of $J^i$. 

Now, since we  have permuted columns in the certain way, the submatrices in subsequent rows must be permuted in the same way as they were in row 1. Hence, all $4\times 4$ matrices are identical to $J^i$.
\end{proof}

\item We proved that for memory-2 strategies, admissible matrices must have a very specific structure. Next, we demonstrate this structure persists with increasing memory.
\begin{lemma}
\label{st: O n consists of blocks}
    The matrix $O$ has the same principal block structure as the matrix in Figure~\ref{fig: matrix $O$}, or the admissible matrices for $N=2$: if we divide $O$ into sixteen equal-sized parts, then only one of the blocks per each row and column will contain nonzero elements. 
\end{lemma}
\begin{proof}
This proof is based on the fact that the permutation $\sigma$ must act in a way that preserves the four-diagonal structure. To demonstrate that the statement of the lemma is true, we need to show that the matrix $O$ does not mix the elements from two different diagonals (however, it can permute the quadruples inside the diagonal and the elements within the quadruples -- as long as the overall structure of the diagonal is preserved). 
We use the first diagonal as an example; the proof can be easily adapted to the remaining three. 

Again, denote the permutation of the numbers $1,\ldots, 2^{2N}$ by $\sigma$. Observe that the element in the upper left corner of $\mathbf{M}(\p,\q)$ (call it $m_{11}$) will become $m_{\sigma(1)\sigma(1)}$ after the conjugation and therefore must belong to the main diagonal of the matrix. 

As demonstrated in Lemma \ref{st: O consists of blocks}, the elements from the first quadruple must lie next to $m_{\sigma(1),\sigma(1)}$ - hence, put in some order, $\sigma(1),\sigma(2),\sigma(3),\sigma(4)$ must be four consecutive numbers. 

Now, $\sigma(2)$ is the row index of the image of the elements of the second quadruple, $\sigma(3)$ of the third and so on. Therefore, the images of the first four rows are consecutive rows. Since the column indices of the elements of the first four quadruples must form a single block as well, the numbers $\sigma(1),\ldots,\sigma(16)$ in some order are 16 consecutive numbers. therefore, the images of the first 16 rows form one block and the numbers $\sigma(1),\ldots,\sigma(64)$ in some order are consecutive numbers, and so on. 

Hence, the image of one diagonal is the entirety of some other diagonal, and the matrix $O$ has to have a form similar to the one in Figure \ref{fig: matrix $O$}.
\end{proof}

\item Observe that the matrices $J_{N-1}^j$ permute the matrices $\overline{M}_i^j$ the same way that their predecessors do at the $N-1$st step - this is clear from writing the transition matrix as in (\ref{eq: memory N-1 matrix}) and multiplying it by $J^i_{N-1}$. 

\item From the previous point we conclude that the matrix $O$ just permutes elements $\overline{M}_i^j$. Therefore, if the matrix $J_i^j$ turns $\overline{M}_i^j$ into $\overline{M}_p^q$, it must 'put' it in its 'appropriate' place in order to form a diagonal - completely analogously to point 2.  

\item From the last point, we can conclude the proof by applying exactly the same reasoning as in the last step of the proof of Lemma \ref{st: symmetries memory 2}; therefore, any matrix $O$ must be a recursively constructed one. 

\item Having demonstrated that the only matrices that \textit{can} conserve that transition matrix structure are $J^i_N$, we demonstrate that the matrices in question \textit{do} indeed possess the required property. 

\begin{lemma}
The matrices $J_N^i$, $i=1,\ldots, 8$, are admissible for memory-$N$ games.
\end{lemma}
\begin{proof}
We demonstrate this by induction. It was already demonstrated that the statement holds for $N=1$ (Lemma~\ref{st: structure conservation memoty 1}). Assume that it holds for $N-1$ and, as per usual, we proceed to show it for $N$. 

Consider a matrix $J^i_{N}$; it is built of 16 $2^{2(N-1)}\times 2^{2(N-1)}$ blocks, the nonzero ones identical to $J_{N-1}^i$; the latter, by assumption, preserves the transition matrix structure of $\mathbf{M}_{N-1}$.

We employ the representation of $\mathbf{M}_{N-1}$ as in (\ref{eq: memory N-1 matrix}). By Lemma \ref{st: O n consists of blocks}, $J_{N-1}^i$ permute the matrices $M_i$; denote this permutation by $\sigma$. 
Therefore, the matrices $\M_i^j$ in $\mathbf{M}_{N}$ undergo the same permutation from the left multiplication by $J^i_{N}$. The four columns of blocks must be permuted accordingly, which is achieved by the  recursive structure of $J^i_N$. 
\end{proof}
\end{enumerate}
These considerations complete our proof of Theorem~\ref{st: symmetries for memory N}. 
\end{proof}


\section{Adaptive dynamics for memory N repeated donation game}
After having described some of the structure of transition matrices of memory-$N$ games, this section is dedicated to investigating the properties of the resulting adaptive dynamics.
That is, we explore what happens when populations of memory-$N$ players are continually moving into the direction of the payoff gradient. 

\label{sec: adaptive dynamics}
\subsection{ $\mathbf{Z}_2$-symmetry}
In general, it is hard to determine whether adaptive dynamics have continuous symmetries. However, for memory-1 strategies, it has been shown there is a $\mathbb{Z}_2$-symmetry (that is to say, the dynamics are invariant under some transformation of order 2):
\begin{theorem}[\cite{laporte2023adaptive}]
    If  $\left(p_{CC}(t),p_{CD}(t),p_{DC}(t), p_{DD}(t)\right)$ is a trajectory of adaptive dynamics for a given repeated donation game, then so is\newline $\left(1-p_{DD}(-t),1-p_{DC}(-t),1-p_{CD}(-t), 1-p_{DD}(-t) \right)$.
\end{theorem}
\noindent
The following theorem generalises the above theorem to memory-$N$ games.
\begin{theorem}
\label{st:exchange}
    If $(p_{CC\ldots C}(t),\ldots, p_{DD\ldots D}(t))$ is a trajectory of the adaptive dynamics for a given donation game, then so is $(1- p_{DD\ldots D}(-t),\ldots, 1-p_{CC\ldots C}(t))$.  
\end{theorem}
The key observation required to prove Theorem \ref{st:exchange} is that when the vectors $\p$ and $\q$ undergo certain transformations, the  game can stay the same through changing the payoff vector and reversing the time. We summarise these statements in the following lemma:

\begin{lemma}
\label{st: identical game}
\begin{enumerate}
\item Time reversibility in a memory-$N$ game can be obtained through reversing the time $t$ and the sign of the payoff vector -- simultaneously done, these two transformations preserve the original game. 
\item Exchanging indices C and D is equivalent to reversing the order in which we write the entries of the payoff vector. Namely, $\mathbf{f}_N\mapsto J_N^8\mathbf{f}_N$. 
\end{enumerate}
\end{lemma}
\begin{proof}
\begin{enumerate}
 \item Reversing time leads to the transformation $\dot{\p}\mapsto-\dot{\p}$. This transformation can be reversed by multiplying the payoff vector by $-1$. 
 \item The second statement is a corollary of Theorem~\ref{thm: symmetry of M} and Lemma~\ref{st: lemma classification}. Exchanging the notions of cooperating and defecting throughout the game means swapping $\p$ for $\mathbf{1}_N - J^8_N\p$ and $\q$ for $\mathbf{1}_N - J^8_N\q$, which is the same as conjugating the matrix $\mathbf{M}_N$ by $J_N^8$. To preserve the payoff function, one needs to multiply $\mathbf{f}_N$ by the same matrix.
 \end{enumerate}
\end{proof}
\noindent
We introduce another transformation of the payoff vector that does not affect the adaptive dynamics:
\begin{lemma}
\label{st: lemma adding vector}
    Adding a vector of the form $C\mathbf{1}_N$ to the payoff vector (i.e., adding a constant~$C$ to all payoffs) does not change the adaptive dynamics.
\end{lemma}
   \begin{remark}
   To stress the dependence of the payoff function  $A(\p,\q)$ on the payoff vector $\mathbf{f}$, in this section we will write $A(\p,\q,\mathbf{f})$.
\end{remark}
\begin{proof}[Proof of Lemma~\ref{st: lemma adding vector}]
   We use the form of the payoff function in (\ref{eq:payoff function}); suppose we tweak it by taking the payoff vector to be $\mathbf{f}_N + C\mathbf{1}_N$. From the linearity of the determinant, 
   \begin{equation*}
   \begin{split}
A(\mathbf{p},\mathbf{q},\mathbf{f}_N + C\mathbf{1}_N) = &\frac{\begin{vmatrix}\widetilde{\mathbf{M}}_N(\mathbf{p},\mathbf{q}) & \mathbf{f}_N + C\mathbf{1}_N\end{vmatrix}}{\begin{vmatrix}\widetilde{\mathbf{M}}_N(\mathbf{p},\mathbf{q}) & \mathbf{1}_N\end{vmatrix}}  =  \frac{\begin{vmatrix}\widetilde{\mathbf{M}}_N(\mathbf{p},\mathbf{q}) & \mathbf{f}_N \end{vmatrix}}{\begin{vmatrix}\widetilde{\mathbf{M}}_N(\mathbf{p},\mathbf{q}) & \mathbf{1}_N\end{vmatrix}} +\\& \frac{\begin{vmatrix}\widetilde{\mathbf{M}}_N(\mathbf{p},\mathbf{q}) & C\mathbf{1}_N\end{vmatrix}}{\begin{vmatrix}\widetilde{\mathbf{M}}_N(\mathbf{p},\mathbf{q}) & \mathbf{1}_N\end{vmatrix}} = A(\p,\q,\mathbf{f}_N) + C.
\end{split}
   \end{equation*}
Therefore, 
\[
\frac{\partial A(\p,\q,\mathbf{f}_N + C\mathbf{1}_N)}{\partial \p}\Big|_{\p=\q} = \frac{\partial A(\p,\q,\mathbf{f}_N)+C }{\partial \p}\Big|_{\p=\q}  = \frac{\partial A(\p,\q,\mathbf{f}_N) }{\partial \p}\Big|_{\p=\q},
\]
and the adaptive dynamics are unchanged.
\end{proof}

\noindent
Based on the three statements above, we are now equipped to prove Theorem~\ref{st:exchange}. 

\begin{proof}[Proof of Theorem \ref{st:exchange}]
We introduce the following transformations:
\begin{equation}
\label{eq: amended game}
    \begin{split}
        &(p_{C\ldots C},\ldots p_{D\ldots D})\to(1-p_{D\ldots D}, \ldots,1-p_{C\ldots C}):=\mathbf{y}\\
         &(q_{C\ldots C},\ldots q_{D\ldots D})\to(1-q_{D\ldots D}, \ldots,1-q_{C\ldots C}):=\mathbf{z}\\
        &G(\mathbf{f}):(f_1,\ldots ,f_{2^{2N}})\mapsto (-f_{2^{2N}},\ldots,-f_1).\\
        &t\to -t =:\kappa
    \end{split}
\end{equation}

\noindent
Lemma \ref{st: identical game} tells us that this game is in fact identical to our original game. Additionally, 
\[
G(\mathbf{f})  = -J^8_N\mathbf{f}_N = \mathbf{f}_N+K_N\mathbf{1},
\]
as per Lemma \ref{st: lemma payoff vector zero sum}.

Lemma \ref{st: lemma adding vector} entails that the vector $K_N\mathbf{1}$ can be disregarded. That is, we can assume that the payoff vector in the amended game (\ref{eq: amended game}) is $\mathbf{f}_N$. 

After these preparations, let us write the adaptive dynamics for (\ref{eq: amended game}). We denote the derivative with respect to $\kappa$ by a dash. The new payoff function is
\[
A_y:=A(\mathbf{y},\mathbf{z},\mathbf{f}_N) = \left<\nu_y, \mathbf{f}_N\right>,
\]
where $\nu_y$ is the left unit eigenvector of the transition matrix of (\ref{eq: amended game}).
Theorem~\ref{st:exchange} states that the transition matrix for the augmented game is $\mathbf{M}(\mathbf{y},\mathbf{z}) = J^8\mathbf{M}(\p,\q)J^8$; hence, the eigenvector $\nu_y = J^8\nu_p$. Bearing this in mind, we write for the adaptive dynamics:
\begin{equation}
\begin{split}
\mathbf{y}' &= \frac{\partial A(\mathbf{y},\mathbf{z},\mathbf{f}_N) }{\partial\mathbf{y}}\Big|_{\mathbf{y} =\mathbf{z}} = \frac{\partial \left<\nu_y,\mathbf{f}_N\right>}{\partial\mathbf{y}}\Big|_{\mathbf{y} =\mathbf{z}} \\&=\frac{\partial\left<J^8\nu_p,\mathbf{f}_N\right>}{\partial\mathbf{p}}\frac{\partial\mathbf{p}}{\partial{\mathbf{y}}}\Big|_{\p=\q} = -(J^8)\frac{\partial\left<\nu_p,J^8\mathbf{f}_N\right>}{\partial\mathbf{p}}\Big|_{\p=\q} \\&=-(J^8)\frac{\partial\left<\nu_p,-\mathbf{f}_N\right>}{\partial\mathbf{p}}\Big|_{\p=\q} =  J^8\dot{\mathbf{p}}.
\end{split}
\end{equation}
For brevity, we denote the right-hand side of the adaptive dynamics equation by $F(\mathbf{p})$, so that $\dot{\p} = F(\p)$. We have demonstrated that 
\begin{equation}
F(\mathbf{y}) = F(\mathbf{1} - J^8\p) = J^8F(\p) =  - D(\mathbf{1} -J^8\p)F(\p),
\end{equation}
where $D(f)$ is the differential of a mapping $f$. 
Therefore, adaptive dynamics are equivariant with respect to the change $\p\mapsto\mathbf{1} -J^8\p$ coupled with the reversal of time.   
\end{proof}

\subsection{Symmetric-anytisymmetric decomposition}
Even for low-dimensional systems, adaptive dynamics can be very hard to tackle. 
To makes some progress, in the following we represent the vector field as the sum of two parts, which we call symmetric and anti-symmetric. While generally, the two parts affect each other, it is possible to choose the parameters in $\mathbf{f}_N$ to view the former as a perturbation of the latter. 

\subsubsection{Reactive strategy}
We start with a motivating example, considering the dynamics of reactive strategies described in Section \ref{sec: reactive strategy}. 
\begin{figure}
    \centering
    \includegraphics[scale=.55]{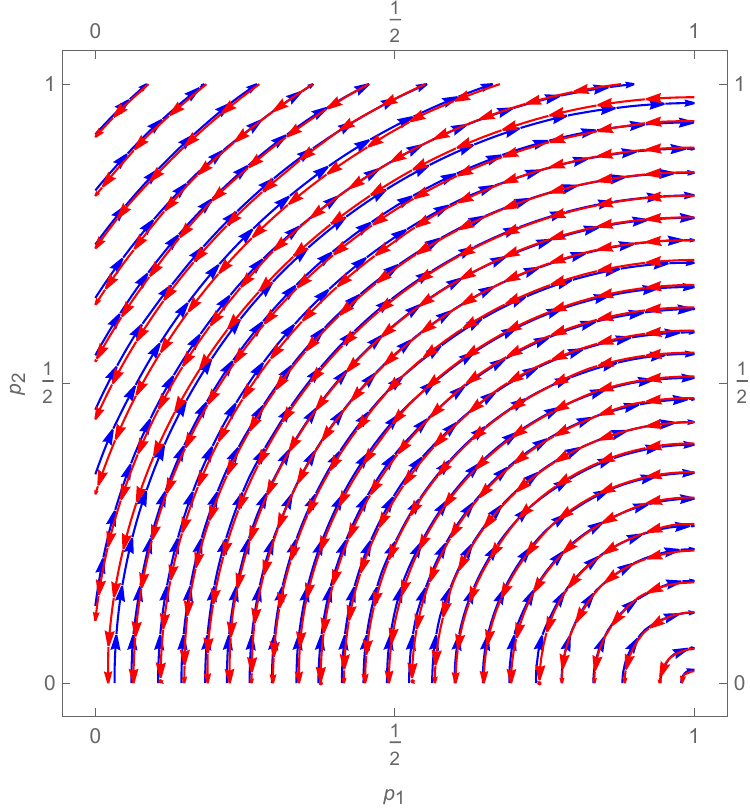}
  \caption{Symmetric and anti-symmetric parts of a reactive strategy. Symmetric part is coloured blue, the anti-symmetric red.}
   \label{fig: sym and antisymm reactive strategy}
\end{figure}
Using the explicit formula \eqref{eq:payff_reactive} for the payoff function $A(p_{1},p_{2},q_{1},q_{2})$, we can write the symmetric and anti-symmetric parts as $$A_a(p_{1},p_{2},q_{1},q_{2}) = \frac12\left(A(p_{1},p_{2},q_{1},q_{2})-A(q_{1},q_{2},p_{1},p_{2})\right)$$ and $$A_s(p_{1},p_{2},q_{1},q_{2}) = \frac12\left(A(p_{1},p_{2},q_{1},q_{2})+A(q_{1},q_{2},p_{1},p_{2})\right), 
$$
respectively. Through these, we obtain:
\begin{equation}
    \label{eq: reactive strategy symmetric anti-symmetric}
    \begin{split}
       A_a &=  \frac{(b+c) \left(p_{2} \left(q_{1}-q_{2}-1\right)-\left(p_{1}-p_{2}\right) q_{2}+q_{2}\right)}{2-2 \left(p_{1}-p_{2}\right) \left(q_{1}-q_{2}\right)},\\
       A_s&=\frac{(b-c) \left(p_{2} \left(q_{1}-q_{2}\right)+\left(p_{1}-p_{2}\right) q_{2}+p_{2}+q_{2}\right)}{2- 2 \left(p_{1}-p_{2}\right) \left(q_{1}-q_{2}\right)}
   \end{split}
\end{equation}
Writing adaptive dynamics for $A_s$ and $A_a$, we get 
\[
\begin{cases}
\dot{p}_1^s = \displaystyle \frac{(b-c) p_{2}}{2 (1-p_{1}+p_{2})^2},\\[0.4cm]
\dot{p}_2^s = \displaystyle  \frac{(b-c) (1-p_{1})}{2 (1-p_{1}+p_{2})^2}
\end{cases}
\]
and 
\[
\begin{cases}
\dot{p}_1^a = \displaystyle \frac{(b+c) p_{2}}{2 \left((p_{1}-p_{2})^2-1\right)},\\[0.4cm]
\dot{p}_2^a = \displaystyle \frac{(b+c) (1-p_{1})}{2 \left((p_{1}-p_{2})^2-1\right)}
\end{cases}
\]
It can be readily checked that both of these vector fields conserve the quantity $(1-p_{1})^2 + p_{2}^2 = C$; however, the flows on the circles go in the opposite direction. We will later observe that this phenomenon persists for memory-2 strategies. 

The adaptive dynamics are plotted in Figure \ref{fig: sym and antisymm reactive strategy}, with symmetric part coloured blue and the anti-symmetric red.
When $c$ is close to 1, the symmetric part may be viewed as a perturbation to the anti-symmetric one. This observation does not yield much information in this case, but a similar one for memory-1 games will be more insightful.

\subsubsection{Memory 1}
\label{sec: decomposition memory 1}
The adaptive dynamics even in the case of memory-1 appear to have a complex structure. However, the dynamics are much easier to study when $b-c=\epsilon\ll 1$, i.e. when mutual cooperation yields a comparably small profit. 
Again, the dynamics can be 'decomposed' into a symmetric and anti-symmetric parts (they are, however, not uncoupled). Then, when the payoff vector $\mathbf{f}^T = \begin{pmatrix}
    b-c&-c&b&0
\end{pmatrix}$, the symmetric and anti-symmetric parts are proportional respectively to $(b-c)$ and $-(b+c)$. 

We know from Theorem \ref{st: decomposition memory N} that the function $A(\p,\q)$ decomposes as the sum of a symmetric and anti-symmteric function, with each differing from the original only in the last column of the determinant in the numerator. Thus, we have 
\begin{small}
\begin{equation}
\begin{split}
    \label{eq: Asym}
A_{s}(\mathbf{p},\mathbf{q})  &:=\frac{\left| 
\begin{array}{cccc}
 p_{CC} q_{CC}-1 & p_{CC}-1 & q_{CC}-1 & f_1 \\
 p_{CD} q_{DC} & p_{CD}-1 & q_{DC} & \frac{f_2 + f_3}{2} \\
 p_{DC} q_{CD} & p_{DC} & q_{CD}-1 & \frac{f_2 + f_3}{2}  \\
 p_{DD} q_{DD} & p_{DD} & q_{DD} & f_4 \\
\end{array}
\right|}{\left| 
\begin{array}{cccc}
 p_{CC} q_{CC}-1 & p_{CC}-1 & q_{CC}-1 & 1 \\
 p_{CD} q_{DC} & p_{CD}-1 & q_{DC} & 1 \\
 p_{DC} q_{CD} & p_{DC} & q_{CD}-1 & 1 \\
 p_{DD} q_{DD} & p_{DD} & q_{DD} & 1 \\
\end{array}
\right|}, \\ \\ \ A_{a}(\mathbf{p},\mathbf{q})  &:=\frac{\left| 
\begin{array}{cccc}
 p_{CC} q_{CC}-1 & p_{CC}-1 & q_{CC}-1 & 0 \\
 p_{CD} q_{DC} & p_{CD}-1 & q_{DC} & \frac{f_2 - f_3}{2} \\
 p_{DC} q_{CD} & p_{DC} & q_{CD}-1 & \frac{f_3 - f_2}{2}  \\
 p_{DD} q_{DD} & p_{DD} & q_{DD} & 0 \\
\end{array}
\right|}{\left| 
\begin{array}{cccc}
 p_{CC} q_{CC}-1 & p_{CC}-1 & q_{CC}-1 & 1 \\
 p_{CD} q_{DC} & p_{CD}-1 & q_{DC} & 1 \\
 p_{DC} q_{CD} & p_{DC} & q_{CD}-1 & 1 \\
 p_{DD} q_{DD} & p_{DD} & q_{DD} & 1 \\
\end{array}
\right|}
\end{split}
\end{equation}
\end{small}

\begin{lemma}
The representation $A(\p,\q) = A_a(\p,\q)+A_s(\p,\q)$ allows us to separate the adaptive dynamics flow into symmetric and anti-symmetric parts. If we define the two parts as 
\[
\dot{\p}_s := \frac{\partial A_s(\p,\q)}{\partial\p}\vert_{\p=\q}, \ \dot{\p}_a:= \frac{\partial A_a(\p,\q)}{\partial\p}\vert_{\p=\q}, 
\]
then $\dot{\p} = \dot{\p}_s(\p_a,\p_s) + \dot{\p}_a(\p_a,\p_s)$.
Additionally, the flow $\dot{\p}_s$ is a gradient one. 
\end{lemma}
\begin{proof}
   In the light of Theorem \ref{st:exchange}, we only need to prove the last statement; one can easily ascertain that 
 \begin{equation}
    \begin{split}
        \frac{\partial A_s(\p,\p)}{\partial \p} &= \lim\limits_{\epsilon\to 0}\frac{A_s(\p+\epsilon,\p+\epsilon)-A_s(\p,\p)} {\epsilon} \\&=\lim\limits_{\epsilon\to 0}\frac{A_s(\p+\epsilon,\p+\epsilon)-A_s(\p, \p+\epsilon) + A_s(\p, \p+\epsilon) -A_s(\p,\p)} {\epsilon} \\ &=\lim\limits_{\epsilon\to 0}\frac{A_s(\p+\epsilon,\p+\epsilon)-A_s(\p, \p+\epsilon) + A_s( \p+\epsilon,\p) -A_s(\p,\p)} {\epsilon} \\
        &=\lim\limits_{\epsilon\to 0}\left(\frac{A_s(\p+\epsilon,\p+\epsilon)-A_s(\p, \p+\epsilon)} {\epsilon}\right)+  \frac{\partial A_s(\p,\q)}{\partial \p}\vert_{\p=\q}\\&=\lim\limits_{\epsilon\to 0}\frac{\partial A_s(\p,\q)}{\partial \p}\vert_{\q= \p+\epsilon}  +  \frac{\partial A_s(\p,\q)}{\partial \p}\vert_{\p=\q} = 2\frac{\partial A_s(\p,\q)}{\partial \p}\vert_{\p=\q},
    \end{split}
\end{equation}
making the flow $\dot{\p}_s$ the flow is the gradient of the function $\frac12 A_s(\p,\p)$.
\end{proof}
\begin{remark}
\begin{enumerate}
\item It is important to remember that the flows $\dot{\p}_a$ and $\dot{\p}_s$ are \textbf{not uncoupled}. Therefore, in the decomposition, they affect each other. 
\item Observe that the first three columns of the matrices in the denominators of $A_a(\p,\q)$ and $A_s(\p,\q)$ are identical to each other and to those of $A(\p,\q)$; the differences lie in the last column, which is $\frac12\left(\mathbf{f} + J_1\mathbf{f}\right)$ for $A_s$ and $\frac12\left(\mathbf{f}-J_1\mathbf{f}\right)$ for $A_a$.   
\end{enumerate}    
\end{remark}

\noindent
The anti-symmetric part of the flow $\dot{\p}_a$ has some interesting properties. On its own, it has the explicit form
\begin{footnotesize}
\begin{equation}
    \label{eq:antisym part}
    \begin{cases}
        \dot{p}_{CC} = \frac{1}{A}p_{DD} (f_2-f_3) (-p_{DD} (p_{CD}+p_{DC})+2 p_{CD} p_{DC}+p_{DD})\\
        \dot{p}_{CD}= -\frac{1}{A}(p_{CC}-1) p_{DD} (f_2-f_3) (p_{CC}-p_{DD}+1)\\
        \dot{p}_{DC}= -\frac{1}{A}(p_{CC}-1) p_{DD} (f_2-f_3) (p_{CC}-p_{DD}+1)\\
        \dot{p}_{DD}=\frac{1}{A}(p_{CC}-1) (f_2-f_3) (p_{CC} (p_{CD}+p_{DC}-1)-2 p_{CD} p_{DC}+p_{CD}+p_{DC}-1)
    \end{cases}
\end{equation}
\end{footnotesize}
where
\begin{footnotesize}
\begin{equation}
\begin{split}
  A &=   2 (p_{CD}-p_{DC}-1) \left(2 p_{DD} \bigl(p_{CC}^2-p_{CD} p_{DC}-1\right)+p_{DD}^2 (-2 p_{CC}+p_{CD}+p_{DC}+1)\\&-(p_{CC}-1) (p_{CC} (p_{CD}+p_{DC}-1)-2 p_{CD} p_{DC}+p_{CD}+p_{DC}-1)\bigr)
    \end{split}
\end{equation}
    \end{footnotesize}
\begin{lemma}
\label{st: invariants antisymm memory 1}
    The flow (\ref{eq:antisym part}) has the following properties: 
    \begin{enumerate}
    \item It has three conserved quantities: 
    \begin{enumerate}
          \item $G_1 = p_{CD}-p_{DC}$,
            \item $G_2 = \frac{1}{3} \left(-p_{CC}^3+3 p_{CC}-3 p_{CD} p_{DC}^2+p_{DC}^3-p_{DD}^3\right)$,
           \item $ G_3 = (1-p_{CC})^2+p_{CD}^2+(1-p_{DC})^2+p_{DD}^2$;
    \end{enumerate}
   
    \item it is bounded;
    \item its trajectories leave  the $(0,1)^4$ cube. 
      \end{enumerate}
\end{lemma}
\begin{proof}
    The statement (1) is proved by explicit computation. One can check that for all $i$ the expression $\left<\nabla G_i, \dot{\p}_a\right> =0$.
    
    Next, (2) follows from (1), since the trajectories necessarily lie on the surface of a four-dimensional sphere. Statement (3) is supported by the fact that $\dot{p}_{CC}<0$ for all values of $p_{CC},p_{CD},p_{DC},p_{DD}$ in the interior of the cube. 
\end{proof}

\noindent
Out of the three conserved quantities, the first and the third are of most interest. The vector field trajectories lie within (but do not coincide, due to singularities!) the intersections of the common level sets of $G_1$, $G_2$ and $G_3$. 
The function $G_1$ describes the difference between $p_{CD}$ and $p_{DC}$; the fact that any level set $p_{CD}-p_{DC} = C$ is conserved (as opposed to $p_{CD}=p_{DC}$ only) is due to the simplified version of the payoff vector. 
Conservation of $G_3$ means that the trajectories lie at a constant distance to the TFT (tit-for-tat) strategy: its description in terms of the probability vector  $\p$ is $(1,0,1,0)$.

\subsubsection{Perturbation dynamics}
Consider the special case of a donation game (with equal gains from switching), such that  $f_1 = b-c, f_2 = -c, f_3 = b, f_4=0$. 
This entails
\begin{small}
\begin{equation}
    \label{eq: Asym for concrete case}
    \begin{split}
A_{s}(\mathbf{p},\mathbf{q})  &=\frac{\left| 
\begin{array}{cccc}
 p_{CC} q_{CC}-1 & p_{CC}-1 & q_{CC}-1 & b-c \\
 p_{CD} q_{DC} & p_{CD}-1 & q_{DC} & \frac{b-c}{2} \\
 p_{DC} q_{CD} & p_{DC} & q_{CD}-1 & \frac{b-c}{2}  \\
 p_{DD} q_{DD} & p_{DD} & q_{DD} & 0 \\
\end{array}
\right|}{\left| 
\begin{array}{cccc}
 p_{CC} q_{CC}-1 & p_{CC}-1 & q_{CC}-1 & 1 \\
 p_{CD} q_{DC} & p_{CD}-1 & q_{DC} & 1 \\
 p_{DC} q_{CD} & p_{DC} & q_{CD}-1 & 1 \\
 p_{DD} q_{DD} & p_{DD} & q_{DD} & 1 \\
\end{array}
\right|},\\ \\ \ A_{a}(\mathbf{p},\mathbf{q})  &=\frac{\left| 
\begin{array}{cccc}
 p_{CC} q_{CC}-1 & p_{CC}-1 & q_{CC}-1 & 0 \\
 p_{CD} q_{DC} & p_{CD}-1 & q_{DC} & -\frac{c+b}{2} \\
 p_{DC} q_{CD} & p_{DC} & q_{CD}-1 & \frac{c+b}{2}  \\
 p_{DD} q_{DD} & p_{DD} & q_{DD} & 0 \\
\end{array}
\right|}{\left| 
\begin{array}{cccc}
 p_{CC} q_{CC}-1 & p_{CC}-1 & q_{CC}-1 & 1 \\
 p_{CD} q_{DC} & p_{CD}-1 & q_{DC} & 1 \\
 p_{DC} q_{CD} & p_{DC} & q_{CD}-1 & 1 \\
 p_{DD} q_{DD} & p_{DD} & q_{DD} & 1 \\
\end{array}
\right|}
\end{split}
\end{equation}
\end{small}
As noted above,  $A_s(\p,\q)$ is directly proportional to $b-c$. Therefore, we may consider the case where $b-c=\epsilon$, a very small quantity. This assumption allows us to interpret the symmetric part of the dynamics as a perturbation to the anti-symmetric one. 
We will leverage the following fact:
\begin{lemma}[\cite{berglund2001perturbation}]
\label{st: lemma perturbation}
    Consider a vector field $\dot{x} = f + \epsilon g$ for small $\epsilon$ in some open subset $U\subset \mathbb{R}^N$. Assume that $f$ admits a uniform Lipschitz constant $K$ in $U$ and that $||g|| $ is bounded by some constant $M$ in the same subset. Then if $x_{\epsilon}(0)= x_0(0) $, where $x_0(t)$ is the trajectory of the system $\dot{x} = f$, then 
    \begin{equation}
        \label{boudedness}
        ||x_{\epsilon}(t)-x_0(t)||\le \frac{\epsilon M}{K}\left(e^{Kt}-1\right)
    \end{equation}
\end{lemma}

\noindent
Our findings will be described for the three-dimensional case of counting strategies (see Section \ref{sec: counting strategy}). The general four-dimensional case can be treated analogously; we provide the corresponding theorem at the end of the section. In the following, we adhere to the notation used in Section \ref{sec: counting strategy} and in \cite{laporte2023adaptive} and denote the symmetric and anti-symmetric components of the counting dynamics as $\dot{\q}_s$ and $\dot{\q}_a$ respectively. 

Under our assumptions, the vector field $\dot{\q}_a$ will take the role of $f$ and $\dot{\q}_s$ will take the role of $\epsilon g$ in the above lemma. However, we encounter the following problem: the denominators of both functions can be zero on the boundary. Hence, we need to multiply everything by the common denominator. 

Firstly, it is obvious that the symmetric part of the dynamics, $\dot{\q}_s$,  multiplied by the denominator, will be a bounded function, since all its modified elements are polynomials. 
The explicit form of the equations is in the Appendix \ref{sec: appendix}, (\ref{eq:anti-symmetric part after multiplication counting}). The only information valuable to us is that all of these equations are polynomials in $q_0,q_{1},q_{2}$.
This immediately entails that $\dot{\q}_a$  is Lipschitz.  This claim follows from the two statements:
\begin{lemma}[Mean value theorem for multivariable functions, \cite{apostol1991calculus}]
    For an open subset $U\subset \mathbb{R}^N$, a differentiable function $f:U\to\mathbb{R}$ and two points $x,y\in U$ the following holds: there exists a $c$ between 0 and 1, such that 
    \[
    |f(y)-f(x)|\le |\nabla f((1-c)x + cy)|||y-x||.
  \]
  \end{lemma}

\begin{remark}
    When all the components of a vecor field are Lipschitz, the vector field is Lipschitz. 
\end{remark}

\noindent
Therefore, both conditions of Lemma \ref{st: lemma perturbation} are satisfied.
However, we are presented with a problem: the skewsymmetric part has zeros on the boundary of the cube.  Specifically, the vector field is 0 when $q_0=1, q_{2}=0$ -- this is one of the edges of the cube. These equilibria are sinks, with two negative and one zero eigenvalues.

The vector field is smooth everywhere, and therefore it would take an infinite time to reach those equilibria - this follows from the existence and uniqueness theorem (see, for example, \cite{arnold2012geometrical}). We overcome this obstacle by narrowing our open set, considering not the entire cube $C$, but the domain $C_{\epsilon_1}$, where we surround the edge $q_{1}=1, q_{2}=0$ by a tube that is $\epsilon$-wide. 
\begin{figure}
    \centering
  \subfigure[Trajectories of the vector field $\dot{\q}_a$ in $(0,1)^3$ ]{\includegraphics[scale =.5]{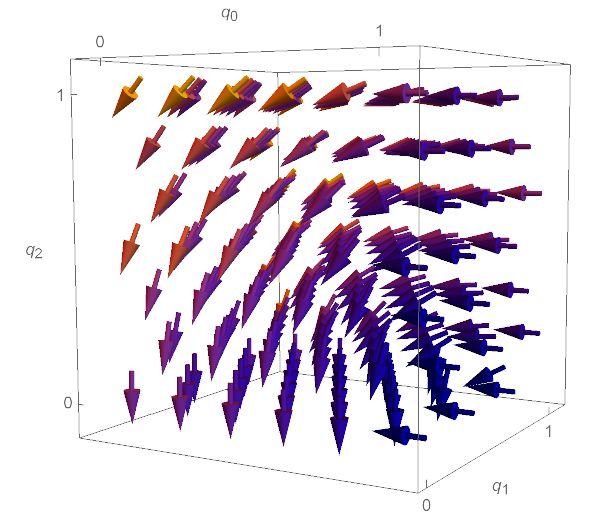}}
  \subfigure[Trajectories of the vector field $\dot{\q}_s$ in $(0,1)^3$]{\includegraphics[scale=.5]{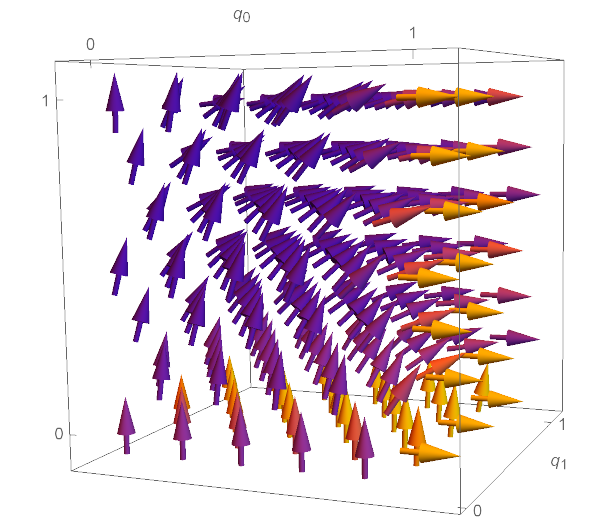}}
  \caption{Symmetric and anti-symmetric parts of the adaptive dynamics flow of counting strategy}
     \label{fig: antisym part}
\end{figure}
It can be checked  computationally that $\dot{q}_0<0$ when $0\le q_0<1, 0\le q_{1}\le1$ and $0<q_{2}\le1$. The same could be said about $\dot{q}_2$. Therefore, the vector field of the anti-symmetric part does not have periodic trajectories and it inevitably escapes the cube $C$.  

The time $T$ that it takes to do so is a smooth function of the starting point. For the cube $C$, this escaping time could be infinite due to the equilibria, but it will always be finite for $C_{\epsilon_1}$.  Moreover, if we take the closure $\overline{C_{\epsilon_1}}$ (by compactifying $C_{\epsilon_1}$, one can determine (and regulate) the maximum escape time $T_{max}$.  With these considerations in mind, we can state:

\begin{theorem}
    For any $\epsilon_1$ and for $b-c<\epsilon(\epsilon_1)$ the trajectories of $\q(t)$ will stay infenitesimally close to those of  $\q_a(t)$.
\end{theorem}
\noindent
In particular, the following holds:
\begin{corollary}
    \begin{enumerate}
\item $\q(t)$ will leave $C_{\epsilon_1, \epsilon}$ in finite time.
\item As a consequence of the previous statement, the vector field $\dot{\q}$ will not have periodic or quasiperiodic trajectories in $C_{\epsilon_1}$;
    \end{enumerate}    
\end{corollary}
\noindent
The dynamics of the anti-symmetric part are depicted in Figure \ref{fig: antisym part}(a).

\subsubsection{Memory $N$}
We can generalize the previous results as follows. 
\begin{theorem}
\label{st: decomposition memory N dynamics}
    For arbitrary memory $N$, the system admits the decomposition into  'symmetric' and 'skew-symmetric' dynamics.  Namely, 
    \begin{equation*}
        \dot{\p}  = \dot{\p}_a + \dot{\p}_s,
    \end{equation*}
    where 
    \begin{equation*}
        \dot{\p}_a = \frac{\partial A_a(\p,\q)}{\partial \p}\vert_{\p=\q}, \ \ \dot{\p}_s = \frac{\partial A_s(\p,\q)}{\partial \p}\vert_{\p=\q}
    \end{equation*}
   with $A_a(\p,\q)$ and $A_s(\p,\q)$ defined in Theorem \ref{st: decomposition memory N}.
Additionally, the flow $\dot{p}_s$ is a gradient flow of the function $\frac12 A_s(\p,\p)$. 
\end{theorem}

\begin{remark}
    As in the case for memory 1, this decomposition is \textbf{not uncoupled}. That is, $\dot{\p}_a = \dot{\p}_a(\p_a,\p_s)$ and $\dot{\p}_s = \dot{\p}_s(\p_a,\p_s)$ and the two respective flows influence each other.
\end{remark}

\noindent
These findings may be interesting for the following reason:
\begin{lemma}
    In the donation game, the vector $\mathbf{f}_N + J_N^2\mathbf{f}_N$ is proportional to $b-c$ and $\mathbf{f}_N - J_N^2\mathbf{f}_N$ is proportional to $-(b+c)$. 
\end{lemma}
\begin{proof}
    The statement follows by induction from the form of $J_N^2$ and the recursive construction of the vector $\mathbf{f}_N$ as in Lemma \ref{st: recursive payoff vector}. 
\end{proof}


\section{Anti-symmetric dynamics}
\label{sec: antisymm dynamics}
The adaptive dynamics for $A_a(\p,\q)$ look very promising, and we believe it warrants a future in-depth investigation,  In the scope of this work we demonstrate a few basic properties. 

Adaptive dynamics for ant--symmetric games are different from those for games with arbitrary payoff vectors. When improving their strategy in accordance with the adaptive dynamics in the anti-symmetric game, the focal player is not trying to maximise their gains: the strategy is honed towards maximising $A(\p,\q)-A(\q,\p)$, i.e. increasing the gap between payoffs. 
It seems natural that the strategy \textit{tit-for-tat} (TFT) will play a special role in the anti-symmetric dynamics, since it equalises the expected payoffs of the two players \cite{press:PNAS:2012}.  Indeed,
\begin{lemma}
For infinitely repeated games without discounting, TFT has the following property:
\[
A_a(\mathbf{p},TFT)=0
\]
for any strategy $\mathbf{p}$ of the co-player. As a consequence, TFT is always a stationary point of the anti-symmetric dynamics.
\end{lemma}
\begin{proof}
By definition of $A_a$, we have
\[
A_a(\mathbf{p},TFT)=\frac{A(\mathbf{p},TFT)-A(TFT,\mathbf{p})}{2}=0,
\]
since $A(\mathbf{p},TFT)=A(TFT,\mathbf{p})$. As a consequence,
\[
\frac{\partial}{\partial\mathbf{p}} A_a(\mathbf{p}, TFT)\vert_{\mathbf{p}=TFT}=0,
\]
which implies TFT is a stationary point of the anti-symmetric dynamics.
\end{proof}

\begin{remark}
    The trajectories of a memory-$N$ anti-symmetric games are not equidistant from TFT: it can be numerically checked that this property fails at $N=2$. 
\end{remark}

The work \cite{laporte2023adaptive} points out that the general adaptive dynamics within $(0,1)^4$ cube can be simplified through re-parameterising the time via multiplication by the determinant. We generalise this statement and demonstrate that the anti-symmetric dynamics have an even simpler form. 
\begin{lemma}
    The denominator 
    \[
    \begin{vmatrix}
        \widetilde{\mathbf{M}}_N(\p,\q) &\mathbf{1}_N
    \end{vmatrix}
    \]
    is a strictly positive function inside $C_N:=(0,1)^{2^{2N}}$ for any $N$. 
\end{lemma}
\begin{proof}
   This is a direct corollary of the Markov theorem \cite{grimmett2020probability}: the determinant $\begin{vmatrix}
        \widetilde{\mathbf{M}}_N(\p,\q) &\mathbf{1}_N
    \end{vmatrix}$ is the scalar product of the eigenvector $\nu$ of $\mathbf{M}_N(\p,\q)$ and the vector $\mathbf{1}_N$; Markov theorem stipulates that all the entries of the stationary distribution $\nu$ are positive. 
\end{proof}
\begin{lemma}
    The vector field $\dot{\p}_a$ has the following form:
    \begin{equation}
   \dot{\p}_a
 = \frac{1}{2\begin{vmatrix}\widetilde{\mathbf{M}}_N(\mathbf{p},\mathbf{p}) &\mathbf{1}_N\end{vmatrix}}
 \left.\frac{\partial\begin{vmatrix}\widetilde{\mathbf{M}}_N(\mathbf{p},\mathbf{q}) & \mathbf{f}_N-J_N^2\mathbf{f}_N \end{vmatrix}}{\partial \p_a}\right|_{\p=\q}
    \end{equation}
\end{lemma}
\begin{proof}
    This immediately follows from the anti-symmetric properties of the numerator and the rules of differentiation: recall that 
 $$
A_a = \frac12 \frac{\begin{vmatrix}\widetilde{\mathbf{M}}_N(\mathbf{p},\mathbf{q}) & \mathbf{f}_N-J_N^2\mathbf{f}_N \end{vmatrix}}{\begin{vmatrix}\widetilde{\mathbf{M}}_N(\mathbf{p},\mathbf{q}) & \mathbf{1}_N\end{vmatrix}}.
$$
    One can easily check that $D_1(\p,\q):=\begin{vmatrix}\widetilde{\mathbf{M}}_N(\mathbf{p},\mathbf{q}) & \mathbf{f}_N-J_N^2\mathbf{f}_N \end{vmatrix}$ is an anti-symmetric function of $\p$ and $\q$, while $D_2(\p,\q):=\begin{vmatrix}\widetilde{\mathbf{M}}_N(\mathbf{p},\mathbf{q}) & \mathbf{1}_N\end{vmatrix}$ is a symmetric one. It immediately follows that 
    \begin{equation}
        \begin{split}
            \dot{\p}_a &= \frac{\partial\left(\frac{D_1(\p,\q)}{D_2(\p,\q)}\right)}{\partial \p}\vert_{\p=\q} = \frac{\frac{\partial D_1(\p,\q)}{\partial \p}\vert_{\p=\q}D_2(\p,\p) - \frac{\partial D_2(\p,\q)}{\partial \p}\vert_{\p=\q}D_1(\p,\p)}{D_2^2(\p,\p)}\\ &= \frac{\frac{\partial D_1(\p,\q)}{\partial \p}\vert_{\p=\q}}{D_2(\p,\p)}.
        \end{split}
    \end{equation}
\end{proof}

\noindent
Therefore, unless the denominator turns 0 (the dynamics are not defined at these points in any case) all the conserved quantities and zeros of the anti-symmetric dynamics coincide with those of the system 
\begin{equation}
    \label{eq: antisymmetric reparanetrisation}
\dot{\p}  = \frac{\begin{vmatrix}\widetilde{\mathbf{M}}_N(\mathbf{p},\mathbf{q}) & \mathbf{f}_N-J_N^2\mathbf{f}_N \end{vmatrix}}{\partial \p}\vert_{\p=\q}.
\end{equation}
Henceforth we will be concerned only with the zeros and concerned quantities  of (\ref{eq: antisymmetric reparanetrisation}).

Another fascinating feature of anti-symmetric dynamics that sets them apart from the general case is that they possess a big number of invariants -- at least $2^{N-1}$ in fact.  Adaptive dynamics with an arbitrary payoff vector  preserve $2^{N-1}$ fixed hyperplanes (see the last theorem in \cite{laporte2023adaptive}); anti-symmetric dynamics generalise this property to fully invariant quantities.

 \begin{proposition}
 \label{st: invariants of the skew symmetric}

 Consider elements $p_{i_1\ldots i_{2N}}$ that encode  games in which the players have been making the same decisions in all rounds except for the first. We denote these  $2(N-1)$ indices $i_3\ldots i_{2N}$ by $I$. 
Then for the anti-symmetric adaptive dynamics the quantities $p_{CD\, I}- p_{DC\, I}$ are conserved. 
In particular, the following functions are invariant:
    \begin{enumerate}
        \item $p_{CDCC\ldots CC}-p_{DCCC\ldots CC}$;
        \item $p_{CDDD\ldots DD}-p_{DCDD\ldots DD}$.
    \end{enumerate}
 \end{proposition}

\noindent
A particular case of this statement was formulated in Lemma~\ref{st: invariants of the skew symmetric} for memory-1 strategies; it can also be numerically checked for memory-2. We provide an analytic proof of the general property, while using the case $N=2$ as an illustration; the statements proved for $N=2$ can be generalised easily.  

Recall that conjugation by the matrix $J^2_N$ exchanges $\p$ and $\q$. We examine the properties of this conjugation for the anti-symmetric dynamics in detail, starting with  \begin{lemma}
\label{st: matrix J eigenvalues}
    The matrix $J^2_N$ has $2^{2N-1} -2^{N-1}$ eigenvectors corresponding to the eigenvalue $-1$ and $2^{2N}-2^{2N-1} + 2^{N-1}$ eigenvectors with eigenvalue 1. 
\end{lemma}
\begin{proof}
Firstly, the matrix $J^2_1$ has eigenvalue $-1$ with eigenvector $\begin{pmatrix} 0&1&-1&0 \end{pmatrix}$ and eigenvalue 1 with eigenvectors $\begin{pmatrix} 1&0&0&0 \end{pmatrix}$,$\begin{pmatrix} 0&0&0&1\end{pmatrix}$,$\begin{pmatrix} 0&1&1&0 \end{pmatrix}$. This is the base of our induction. 

Consider the matrix $J^2_{N-1}$ and let it have $k$ eigenvectors $\mathbf{u}_1,\ldots, \mathbf{u}_k$ corresponding to eigenvalue $-1$ and $2^{2(N-1)}-k$ eigenvectors corresponding to the eigenvalue 1. Note that we are making two assumptions here: on the total number of eigenvectors and on their parity. 
Suppose that a vector $\mathbf{v}$ is such that $J^2_{N-1}\mathbf{v} =\mathbf{v}$ and a vector $\mathbf{w}$ is such that $J^2_{N-1}\mathbf{w} = -\mathbf{w}$.  Then the following vectors are eigenvectors of $J^2_N$ with eigenvalue 1:
\[
\begin{pmatrix}
    \mathbf{v}\\
    0\\
    0\\
    0\\
\end{pmatrix}, \ \ \begin{pmatrix}
0\\
    \mathbf{v}\\
    \mathbf{v}\\
     0\\
\end{pmatrix},  \ \ \begin{pmatrix}
    0\\
    \mathbf{w}\\
    -\mathbf{w}\\
    0\\
\end{pmatrix}, \ \  \begin{pmatrix}
    0\\
    0\\
    0\\
    \mathbf{v}
\end{pmatrix},
\]
and the following vectors correspond to eigenvalue $-1$:
\[
\begin{pmatrix}
    \mathbf{w}\\
    0\\
    0\\
    0\\
\end{pmatrix}, \ \ \begin{pmatrix}
0\\
    \mathbf{w}\\
    \mathbf{w}\\
     0\\
\end{pmatrix},  \ \ \begin{pmatrix}
    0\\
    \mathbf{v}\\
    -\mathbf{v}\\
    0\\
\end{pmatrix}, \ \  \begin{pmatrix}
    0\\
    0\\
    0\\
    \mathbf{w}
\end{pmatrix}
\]
Adopting our induction assumption, we can conclude that the set of newly constructed vectors contains $2^{2N}$ vectors, of which $2k + 2^{2(N-1)}$ have eigenvalue $-1$ and $3*2^{2(N-1)} - 2k$ eigenvalue $1$. 

The last step is the computation of the sum 
\begin{equation}
\label{eq: sum of 2s}
2^{2(N-1)} + 2\left(2^{2(N-2)} + 2\left(2^{2(N-3)} + 2\left(2^{N-4} + 2(\ldots + 2(2^0))\right)\right)\right)
\end{equation}
for the total number of $-1$ eigenvectors. One can see that 
\begin{equation*}
    \label{eq: computation}
    (\ref{eq: sum of 2s}) = 2^{2N-2} + 2^{2N-4} +\ldots  + 2^{N-1} = 2^{N-1}(2^N-1)  =2^{2N-1} - 2^{N-1}, 
\end{equation*}
which concludes our proof.
\end{proof}

\noindent
By construction, $A_a(\p,\q)$ is an anti-symmetric function of $\p$ and $\q$. 
\begin{lemma}
If a vector $\phi_N\in \mathbf{R}^{2^{2N}}$ is such that $J^2_N\phi = -\phi$, then for the  matrix $\begin{pmatrix}\widetilde{\mathbf{M}}_N(\mathbf{p},\mathbf{p})&\phi_N \end{pmatrix}$ the following hold:
    \begin{enumerate}
     \item The matrix $\begin{pmatrix}\widetilde{\mathbf{M}}(\mathbf{p},\mathbf{p}) & \phi_N  \end{pmatrix}$ has linearly dependent columns.
    \item $J^2\begin{pmatrix}\widetilde{\mathbf{M}}_N(\mathbf{p},\mathbf{p}) & \phi_N \end{pmatrix}J^2 =\begin{pmatrix}\widetilde{\mathbf{M}}_N(\mathbf{p},\mathbf{p}) &-\phi_N \end{pmatrix}$.
   
\end{enumerate}
\end{lemma}

\begin{proof}
Figure \ref{fig:mat degenerate} shows the matrix $\begin{pmatrix}\widetilde{\mathbf{M}}_N(\p,\p), \phi_N\end{pmatrix}$ for $N=2$ and $\phi$ for the game with equal gains from switching - a computation yields that the rank of this $16\times 16$ matrix is 15. 

For the general case, the function $\begin{vmatrix}\widetilde{\mathbf{M}}_N(\p,\q), \phi_N\end{vmatrix}$ is an anti-symmetric function of $\p$ and $\q$ due to our choice of $\phi$. Therefore setting $\q=\p$ turns it 0, which entails that its columns form a linearly dependent set. Note that this property holds for any $\phi_N$, as long as $J^2\phi_N = -\phi_N$. 
\begin{figure}
    \centering
    \includegraphics[scale=.21]{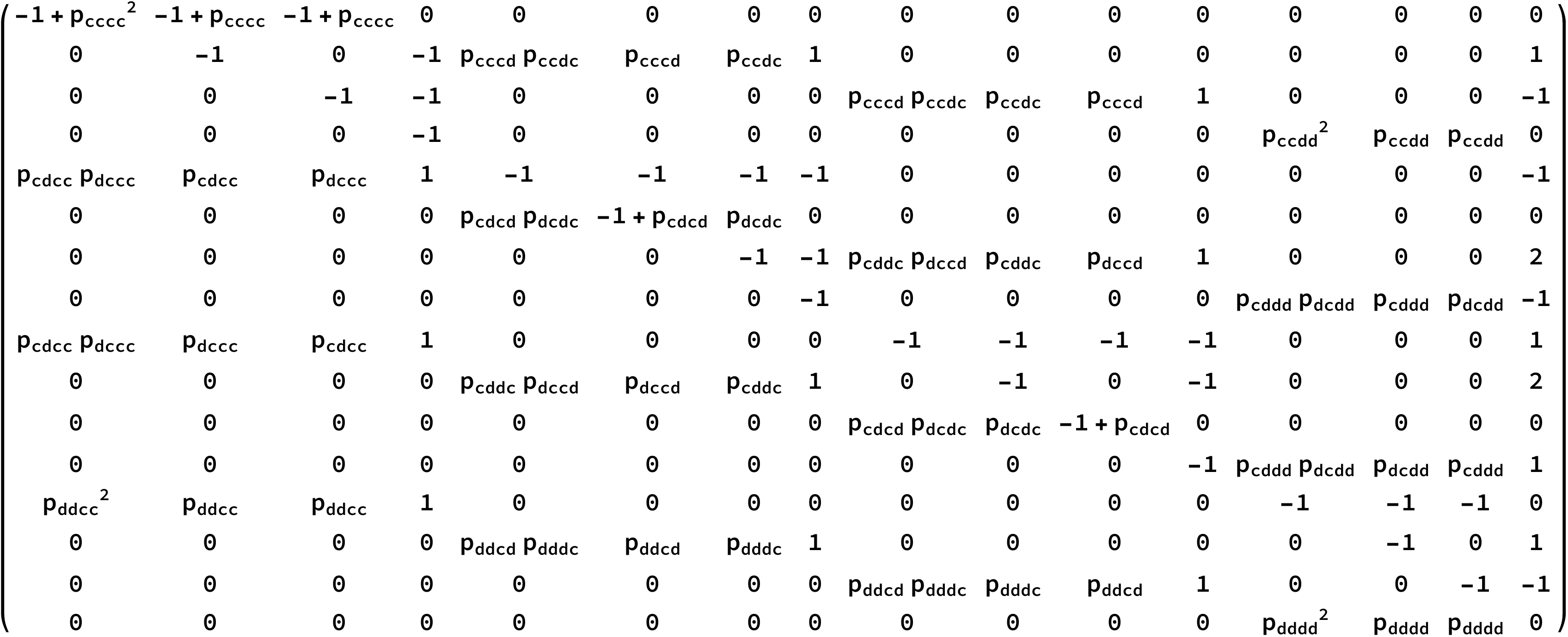}
    \caption{The matrix $\begin{pmatrix}\widetilde{\mathbf{M}}(\mathbf{p},\mathbf{p}) & \mathbf{f}-J_N^2\mathbf{f}\end{pmatrix}$}
    \label{fig:mat degenerate}
\end{figure}
The second statement follows from the fact that conjugation by $J^2_N$ exchanges $\p$ and $\q$; therefore, the matrix $\begin{pmatrix}\widetilde{\mathbf{M}}_N(\p,\p), \phi_N\end{pmatrix}$ has to be invariant under it, barring the last column. 
\end{proof}

\begin{lemma}
\label{st: permutation skew sym }
Left multiplication by the matrix $J^2_N$ permutes the  columns of $\begin{pmatrix}\widetilde{\mathbf{M}}_N(\p,\p), \phi_N\end{pmatrix}$ (and multiplies the last one by $-1$). 
If we number the columns from left to right and denote this permutation by $\mu$, then this permutation leaves $2^{N}$  numbers intact. In particular, 
\begin{equation}
\label{eq: permutation skew sym }
    \begin{split}
        &\mu(1) = 1,\\
       & \mu(2)=3,\\
        &\mu(3) = 2,\\
        &\mu(4) = 4, 
    \end{split}  \ \ \ \ \ \ \ \ \ \ \
    \begin{split}
&\mu(2^{2N}-4) = 2^{2N}-4, \\ &\mu(2^{2N}-3) = \mu(2^{2N}-2),\\ & \mu(2^{2N}-2) = \mu(2^{2N}-3), \\ & \mu(2^{2N}) = \mu(2^{2N}).
    \end{split}
\end{equation}

\end{lemma}
\begin{proof}

    Consider the expression $J^2 \begin{pmatrix}\widetilde{\mathbf{M}}_N(\mathbf{p},\mathbf{p}) & \phi \end{pmatrix}J^2$.  As stated in Lemma \ref{st: permutation skew sym }, this matrix is equal to  $\begin{pmatrix}\widetilde{\mathbf{M}}_N(\mathbf{p},\mathbf{p}) & -\phi \end{pmatrix}$ 
    
    Multiplication by $J^2_N$ from the left permutes the rows of $\begin{pmatrix}\widetilde{\mathbf{M}}(\mathbf{p},\mathbf{p}) & \phi\end{pmatrix}$ while multiplication from the right permutes its columns. The last column is turned into minus itself, but barring it the set of columns of $\begin{pmatrix}\widetilde{\mathbf{M}}(\mathbf{p},\mathbf{p}) & \phi\end{pmatrix}$ must be invariant under left multiplication by $J^2_N$; hence it is a permutation of this set. 

The relations (\ref{eq: permutation skew sym }) are clear for $N=2$; for larger $N$, we proceed by induction. Suppose the statement is true for some $N$, i.e. $2^{2N}$ columns remain in their original places after multiplication from the left by $J_N^2$. Then for $N+1$, the number of stationary vectors will be twice as large -- this follows from the construction of the matrix $J^2_{N+1}$, namely, from the top right and bottom left copies of $J^2_{N-1}$. Therefore, the total number of stationary vectors is $2^{N+1}$.

In particular, equations (\ref{eq: permutation skew sym }) follow directly from the form of $J_N^2$.

\end{proof}

Once again, one can easily check these lemmata on  the toy example from Figure \ref{fig:mat degenerate}. If we denote the  columns of $\begin{pmatrix} \widetilde{\mathbf{M}}_2(\p,\p), \mathbf{f}_2-J^2_2\mathbf{f}_2\end{pmatrix}$ by $\v_1, \ldots \v_{16}$, the explicit form of $\mu$ is  
\begin{equation}
\begin{split}
J^2\v_1 &= \v_1\\
J^2\v_2 &= \v_3\\
J^2\v_3 &= \v_2\\
J^2\v_4 &= \v_4\\
\end{split}\ \ \ 
\begin{split}
J^2\v_5 &= \v_9\\
J^2\v_6 &= \v_{11}\\
J^2\v_7 &= \v_{10}\\
J^2\v_8 &= \v_{12}\\
\end{split}\ \ \ 
\begin{split}
&J^2\v_9 = \v_5\\
&J^2\v_{10} = \v_7\\
&J^2\v_{11} = \v_6\\
&J^2\v_{12} = \v_8\\
\end{split}\ \ \ 
\begin{split}
&J^2\v_{13} = \v_{13}\\
&J^2\v_{14} = \v_{15}\\
&J^2\v_{15} = \v_{14}\\
&J^2\v_{16} = -\v_{16}\\
\end{split}
\end{equation}
In general, we denote this permutation by $\sigma$ and write that $J\v_{i} = \v_{\sigma(i)}$
\begin{equation}
\label{eq: zero combination}
k_1\v_1 + k_2\v_2+\ldots +k_{2^{2N}}\v_{2^{2N}} = \mathbf{0}.
\end{equation}
Multiplying  (\ref{eq: zero combination}) by $J^2_N$ gives
\begin{figure}
    \centering
    \includegraphics[scale=.31]{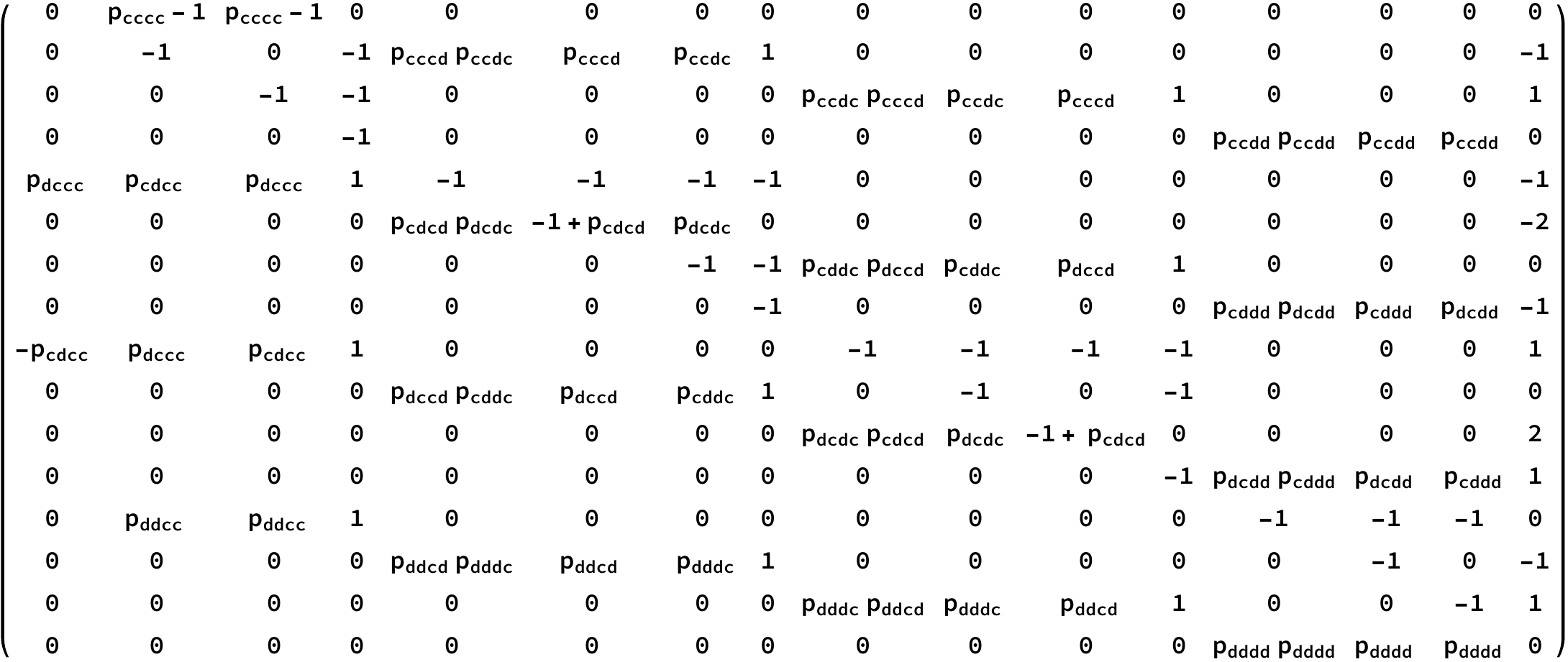}
    \caption{The first part of the directional derivative of $p_{}$  }
    \label{fig:der1}
\end{figure}
\begin{equation}
\label{eq: another zero sum}
\begin{split}
k_1\v_{\sigma(1)} + k_2\v_{\sigma(2)}+\ldots -k_{2^{2N}}\v_{\sigma(2^{2N})} = \mathbf{0}.
\end{split}
\end{equation}
Subtracting (\ref{eq: another zero sum}) from (\ref{eq: zero combination}), we get 
\begin{equation}
\label{eq: third zero sum}
\begin{split}
\sum\limits_{i=1}^{2^{2N}-1} \v_i\left(k_i-k_{\sigma(i)}\right) + 2k_{2^{2N}}\v_{2^{2N}}=0.
\end{split}
\end{equation}
As per Lemma \ref{st: permutation skew sym }, $2^N$  vectors,  including $\v_1, \v_4, \v_{2^{2N}-4}$ are included in this sum with zero coefficients. There are two possibilities:
\begin{enumerate}
    \item  the vectors in (\ref{eq: third zero sum}) are linearly independent, and all the coefficients are equal to 0, including $k_{2^{2N}}$;
    \item (\ref{eq: third zero sum}) has at least one nonzero coefficient. 
\end{enumerate}

\noindent
Option 1 being true entails that the matrix $\begin{pmatrix}\widetilde{\mathbf{M}}(\mathbf{p},\mathbf{p}) & \mathbf{f}\end{pmatrix}$ is degenerate for every choice of the vector $\mathbf{f}$ (since $k_{2^{2N}}=0$, (\ref{eq: zero combination}) holds for all the columns of $\mathbf{M}_N(\p,\p)$ for any payoff vector), which is not true (consider, e.g. the symmetric dynamics). Hence, there are nonzero coefficients in (\ref{eq: third zero sum}).

\begin{observation}
\label{st: observation skew symm eigenpairs}
    The following hold true:
    \begin{enumerate}
        \item 
        The vectors $\mathbf{v}_i$ and $\mathbf{v}_{\sigma(i)}$ \textbf{do} pair up nicely as in \ref{eq: third zero sum} (i.e. there are no cycles of length 3 or  
        more) owing to the fact that $J^2_N$ has order 2.
        \item As per Lemma \ref{st: permutation skew sym }, there are $2^{2N-1}-2^{N-1}$ of these pairs of vectors. 
        \item All the vectors $\mathbf{v}_i-\mathbf{v}_{\sigma(i)}$ are eigenvectors of the matrix $J^2_N$ corresponding to eigenvalue $-1$. 
    \end{enumerate}
\end{observation}
\noindent
Armed with these preliminary results,  we can move on to proving Proposition \ref{st: invariants of the skew symmetric}. 

\begin{proof}[Proof of Proposition~\ref{st: invariants of the skew symmetric}]
 We use memory-2 matrices and taking derivatives with respect to $p_{CDCC}$ and $p_{DCCC}$ as an example, but this proof can be repeated verbatim for the second invariant and for any $N$ -- it will be pointed out how crucial steps of the proof can be generalised. Lastly, we will demonstrate how the proposition can be generalised to any arbitrary index $I$. 

Consider the directional derivative of the function $p_{CDCC}-p_{DCCC}$ along the flow of the vector field $\dot{\p}_a$. Using the rules for differentiation of determinants and setting $\p=\q$, we write 
\begin{equation}
\begin{split}
\partial_{\dot{\p}_a}\left(p_{CDCC}-p_{DCCC}\right) &=\left<\nabla\left(p_{CDCC}-p_{DCCC}\right) , \dot{\p}_a\right>\\& =  \frac{1}{\begin{vmatrix}\widetilde{\mathbf{M}}_N(\mathbf{p},\mathbf{q}) & \mathbf{1}_N\end{vmatrix}}\left(\det(M_1) + \det(M_2)\right)
\end{split}
\end{equation}
where the matrices $M_1$ amd $M_2$ are respectively shown in Figures \ref{fig:der1} and \ref{fig:der2}.
\begin{figure}
    \centering
    \includegraphics[scale=.6]{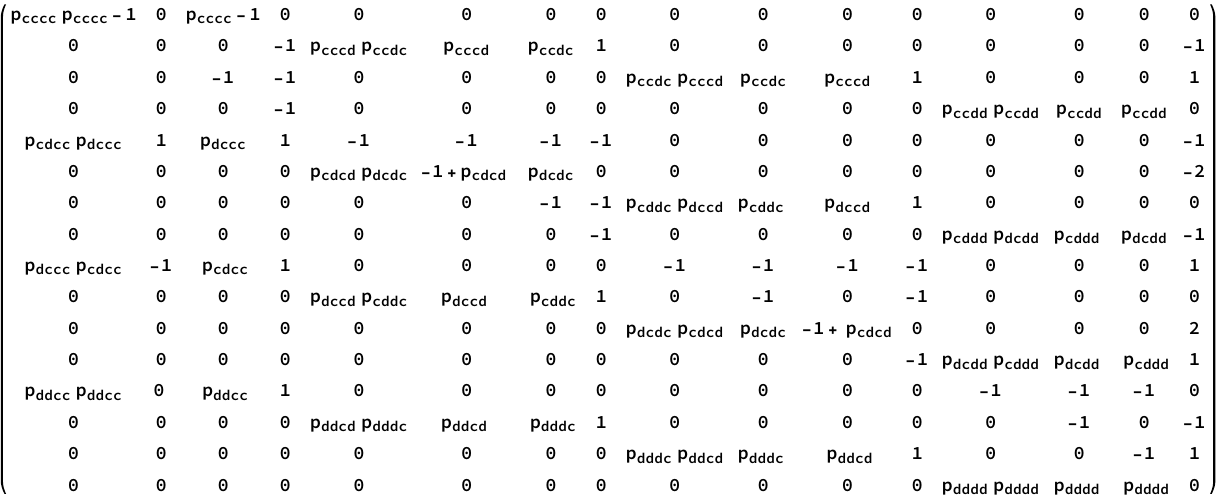}
    \caption{The second part of the directional derivative of $p_{}$  }
    \label{fig:der2}
\end{figure}
Now we need to demonstrate that the determinants of both matrices are equal to zero. We start with the first. 

$\mathbf{det(M_1)=0}$. This is the more easily proved statement of the two: all the columns of $M_1$ except for the first are identical to those of $\mathbf{M}_N(\p,\p)$, and the columns of $\mathbf{M}(\p,\p)$ form a linearly dependent set even without the first one, so it does not matter which form the first column has.

$\mathbf{det(M_2)=0}$. Here, we leverage on Lemma \ref{st: matrix J eigenvalues} and Observation \ref{st: observation skew symm eigenpairs}. As stated there, the vectors $\mathbf{v}_i-\mathbf{v}_{\sigma(i)}$ and $\mathbf{v}_{2^{2N}}$  are eigenvectors of $J^2_N$ with eigenvalue $-1$; moreover, there are $2^{2N-1}-2^{N-1}+1$ such vectors, whereas the $-1$ eigenspace of $J^2_N$ has dimension  $2^{2N-1}-2^{N-1}$; therefore, this system has to be linearly dependent. 

Consider the following set of vectors: $\mathbf{v}_{16}$, $\mathbf{v}_i-\mathbf{v}_{\sigma(i)}$ for $i=3,\ldots, 15$  and the second column of $M_2$, i.e. the vector 
\[
\begin{pmatrix}
0 &
 0 &
 0 &
 0 &
 1 &
 0 &
 0 &
 0 &
 -1 &
 0 &
 0 &
 0 &
 0 &
 0 &
 0 &
 0 &
 \end{pmatrix}
\]
(for arbitrary $N$, take the vector with $1$ in the same row as $p_{CDCC\ldots CC}$ and $-1$ in the same row as $p_{DCCC\ldots CC}$; and for the second function from Lemma \ref{st: invariants of the skew symmetric}, take a vector with $1$ in the same row as $p_{CDDD\ldots DD}$ and $-1$ in the same row as $p_{DCDD\ldots DD}$). The span of this set is included in the span of the columns of $M_2$.

This set  contains 7 vectors ($2^{2N-1}-2^{N-1}+1$ for generic values of  $N$), whereas there can be at most 6 (or $2^{2N-1}-2^{N-1}$) linearly independent eigenvectors of $J^2_N$ with eigenvalue $-1$. Hence, our constructed set is linearly dependent, and so are the columns of $M_2$, making its determinant 0.

\textbf{General case.} Generalising the proof above  from $p_{CDCC\ldots CC} -p_{DCCC\ldots CC}$ to $p_{CDDD\ldots DD} -p_{DCDD\ldots DD}$ is straightforward, but it is less obvious why the conservation property holds for an arbitrary set of indices $I$. 

For any $I$ as in Proposition \ref{st: invariants of the skew symmetric}, consider the 4 columns of the matrix $\begin{pmatrix}\widetilde{\mathbf{M}}_N(\p,\q) & \mathbf{f}\end{pmatrix}$ where the elements $p_{CD \ I}$ and $p_{DC \ I}$ appear - observe that due to the choice of $I$, these elements do indeed belong to four neighbouring columns. The matrix has the following form:

\begin{equation}
\begin{pmatrix}\widetilde{\mathbf{M}}_N(\p,\q) & \mathbf{f}\end{pmatrix} =  \begin{pmatrix}
        \vdots & \vdots &\vdots & \vdots& \vdots\\
        \vdots & p_{CD \ I}q_{DC\ I}& p_{CD\ I}&q_{DC\ I}&\vdots\\
        \vdots & \vdots &\vdots & \vdots& \vdots\\
        \vdots & p_{DC \ I}q_{CD\ I}& p_{DC\ I}&q_{CD\ I}&\vdots\\
        \vdots & \vdots &\vdots & \vdots& \vdots\\
    \end{pmatrix}
    \end{equation}
After differentiating and setting $\p=\q$, we get 
\begin{small}
\begin{equation}
\begin{split}
\left<\frac{\partial \begin{pmatrix}\widetilde{\mathbf{M}}_N(\p,\q) & \mathbf{f}\end{pmatrix}}{\partial \p}\vert_{\p=\q} , \nabla\left( p_{CD\ I}-p_{DC \ I}\right)\right>&= \underbrace{\begin{pmatrix}
        \vdots & \vdots &\vdots & \vdots& \vdots\\
        \vdots & p_{DC\ I}& p_{CD\ I}&p_{DC\ I}&\vdots\\
        \vdots & \vdots &\vdots & \vdots& \vdots\\
        \vdots & -p_{CD\ I}& p_{DC\ I}&p_{CD\ I}&\vdots\\
        \vdots & \vdots &\vdots & \vdots& \vdots\\
    \end{pmatrix}}_{\mathbf{M}_1}\\&+ \underbrace{\begin{pmatrix}
        \vdots & \vdots &\vdots & \vdots& \vdots\\
        \vdots & p_{CD \ I}p_{DC\ I}& 1&p_{DC\ I}&\vdots\\
        \vdots & \vdots &\vdots & \vdots& \vdots\\
        \vdots & p_{DC \ I}p_{CD\ I}& -1&p_{CD\ I}&\vdots\\
        \vdots & \vdots &\vdots & \vdots& \vdots\\
    \end{pmatrix}}_{\mathbf{M}_2}
    \end{split}
\end{equation}
\end{small}
\begin{remark}
    Bear in mind that there are "$-1$"s "concealed" in the dots - we are omitting them for the sake of brevity.
\end{remark}
\noindent
The matrices $\mathbf{M}_1$ and $\mathbf{M}_2$ are different from $\begin{pmatrix}\widetilde{\mathbf{M}}_N(\p,\p) & \mathbf{f}\end{pmatrix}$ in one column each.
Observe that for $\mathbf{M}_1$ the different column  is $\begin{pmatrix} \ldots &p_{DC \ I}&\ldots &-p_{CD \ I}&\ldots \end{pmatrix}$ replacing  $\begin{pmatrix} \ldots& p_{DC \ I}& p _{DC \ I}&\ldots p_{CD \ I} &p_{DC\ I}&\ldots \end{pmatrix}$. The latter is an eigenvector of $J^2_N$ with eigenvalue 1; therefore, the column system of the matrix $\mathbf{M}_1$ still contains all $2^{2N-1} - 2^{N-1}$ pairs of vectors $\mathbf{v}_i-\mathbf{v}_{\sigma(i)}$ plus plus the vector $\mathbf{f}$ -i.e. again, one more vector  than the dimension of the $-1$ eigenspace of $J^2_N$, and hence a linearly dependent subsystem. This entails that $\det \mathbf{M}_1 =0$. 

To show that $\det \mathbf{M}_2 = 0$, note that the $\begin{pmatrix}\ldots &1&\ldots&-1
&\ldots\end{pmatrix}$ is an eigenvector of $J^2_N$ with $-1$ eigenvalue; thus, the number of $-1$ eigenvectors of $J^2_N$ within the column system of $\mathbf{M}_2$ is still $2^{2N-1} - 2^{N-1}+1$, once again making it a linearly dependent system.

 \end{proof}
\section{Conclusion}
In the scope of this paper, we examined some of the properties of the repeated donation game from the point of view of linear algebra and dynamical systems. 
In particular, we demonstrated that a classification of transition matrices can be achieved, with equality classes consisting of matrices conjugated by eight permutation matrices. We described what changes in the "point of view" within the game these matrices describe and examined the more interesting ones in detail: those that exchange the notions of cooperation and defection and change the focal player. 

Both of these transformations offer insights into the properties of the memory-$N$ game: the former gives rise to a $\mathbf{Z}_2$-symmetry and the latter allows to decompose the dynamics into symmetric and anti-symmetric parts and treat the first as a perturbation to the second. 

The adaptive dynamics of the anti-symmetric part, when considered on their own, look extremely promising: they generalise the invariant properties of the setup with the arbitrary payoff vector and have a fascinating interpretation as strategies that aim not to increase the absolute profit, but to widen the "wealth gap" between the two players. 

We believe that the approaches used in this paper shed some light on the properties of memory-$N$ donation games, as well as their adaptive dynamics and provide a pure mathematics perspective on the problem. Additionally, the authors believe that the games with an anti-symmetric payoff vector show a lot of promise and are worth investigating in more detail. 

\appendix
\section{Explicit formulae}
\label{sec: appendix}
Explicit equations for memory 1 adaptive dynamocs for an arbitrary payoff vector $\left(f_1\ f_2 \ f_3 \ f_4\right)$:
\begin{footnotesize}
    \begin{equation}
        \label{eq: adaptive dynamics general form}
\begin{cases}
 \dot{p}_{CC} =&\frac{1}{A} \Biggl(p_{DD} (2 p_{CD} p_{DC}-(p_{CD}+p_{DC}) p_{DD}+p_{DD}) \Bigl(-f_3 p_{CC}^2+f_3 p_{CD} p_{CC}^2\\&-f_1 p_{CD}^2 p_{CC}+f_1 p_{DC}^2 p_{CC}-f_1 p_{CC}+f_3 p_{CC}+2 f_1 p_{CD} p_{CC}-f_3 p_{CD} p_{CC}\\&+f_3 p_{DC} p_{CC}-f_3 p_{CD} p_{DC} p_{CC}-f_1 p_{CD} p_{DC}^2+(f_3 (p_{CC}-p_{DC}-1)\\&+f_1 (-p_{CD}+p_{DC}+1)) p_{DD}^2+f_1 p_{CD}^2 p_{DC}-f_3 p_{DC}-f_1 p_{CD} p_{DC}+f_3 p_{CD} p_{DC}\\&-f_4 (p_{CD}-p_{DC}-1) \bigl(p_{CC}^2-(p_{CD}+p_{DC}+1) p_{CC}+p_{CD} p_{DC}+1\bigr)\\&-\bigl(f_3(p_{CC} (p_{CC}+p_{CD}-1) -(p_{CC}+p_{CD}) p_{DC} -1)\\&-2 f_1 p_{CC} (p_{CD}-p_{DC}-1)\bigr) p_{DD}+f_2 \bigl((p_{CD}-p_{CC}) p_{DD}^2\\&+\left(p_{CC}^2+(-p_{CD}+p_{DC}+1) p_{CC}-p_{CD} p_{DC}-1\right) p_{DD}\\&-(p_{CC}-1) (-p_{DC} p_{CD}+p_{CD}+p_{CC} p_{DC}-1)\bigr)\Bigr)\Biggr)\\
 \dot{p}_{CD}=&-\frac{1}{A}(p_{CC}-1) (p_{CC}-p_{DD}+1) p_{DD} \Biggl((f_3 (p_{CC}-p_{DC}-1)\\&+f_1 (-p_{CD}+p_{DC}+1)) p_{DD}^2+\bigl(2 f_1 (p_{CD}-p_{DC}-1) p_{DC}\\&+f_3 \left(-p_{CC}^2+p_{DC}^2+p_{DC}+1\right)\bigr) p_{DD}\\&+f_4 (p_{CD}-p_{DC}-1) \left((p_{CC}-p_{DC})^2+p_{DC}-1\right)\\&+p_{DC} \Bigl(f_3 (p_{CC}-1) (p_{CC}-p_{DC})+f_1 \left(p_{DC}^2-p_{CD} p_{DC}+p_{CD}-1\right)\Bigr)\\&+f_2 \Bigl((p_{CD}-p_{CC}) p_{DD}^2+\left(p_{CC}^2+p_{DC}^2-2 p_{CD} p_{DC}+p_{DC}-1\right) p_{DD}\\&-(p_{CC}-1) \left(p_{DC}^2-2 p_{CD} p_{DC}+p_{DC}+p_{CC} (p_{CD}-1)+p_{CD}-1\right)\Bigr)\Biggr)\\
 \dot{p}_{DC}=& \frac{1}{A}(p_{CC}-1) (p_{CC}-p_{DD}+1) p_{DD} \Biggl(f_1 p_{CD}^3-2 f_1 p_{CD}^2+f_3 p_{CD}^2\\&-f_3 p_{CC} p_{CD}^2-f_1 p_{DC} p_{CD}^2+f_1 p_{CD}-f_3 p_{CD}+f_3 p_{CC} p_{CD}+f_1 p_{DC} p_{CD}\\&-2 f_3 p_{DC} p_{CD}+2 f_3 p_{CC} p_{DC} p_{CD}+(f_1 (p_{CD}-p_{DC}-1)\\&+f_3 (-p_{CC}+p_{DC}+1)) p_{DD}^2-f_4 \left((p_{CC}-p_{CD})^2+p_{CD}-1\right) (p_{CD}-p_{DC}-1)\\&-f_3 p_{CC}^2 p_{DC}+f_3 p_{DC}+\bigl(2 f_1 p_{CD} (-p_{CD}+p_{DC}+1)\\&+f_3 \left(p_{CC}^2+p_{CD} (p_{CD}-2 p_{DC}-1)-1\right)\bigr) p_{DD}+f_2 \bigl((p_{CD}-p_{DD}-1) p_{CC}^2\\&+\left(-p_{CD}^2+p_{CD}+p_{DD}^2\right) p_{CC}-2 p_{CD}+p_{DD}+p_{CD} (p_{CD}-p_{DD}) (p_{DD}+1)+1\bigr)\Biggr)\\
 \dot{p}_{DD}=&-\frac{1}{A}(p_{CC}-1) (-2 p_{DC} p_{CD}+p_{CD}+p_{DC}+p_{CC} (p_{CD}+p_{DC}-1)-1) \\&*\Biggl((f_3 (p_{CD}-p_{CC})+f_1 (-p_{CD}+p_{DC}+1)) p_{DD}^2+\Bigl(f_3 (p_{CC}-1) (p_{CC}-p_{CD}+2)\\&+f_3 (p_{CC}-p_{CD}-1) p_{DC}+f_1 \left((p_{CD}-1)^2-p_{DC}^2\right)\Bigr) p_{DD}\\&+p_{DC} (f_1 p_{CD} (-p_{CD}+p_{DC}+1)-f_3 (p_{CC}-1) (p_{CC}-p_{CD}+1))\\&-f_4 (p_{CD}-p_{DC}-1) \left(p_{CC}^2-2 p_{DD} p_{CC}-p_{CD} p_{DC}+(p_{CD}+p_{DC}+1) p_{DD}-1\right)\\&+f_2 \bigl((p_{CD}-p_{DD}-1) p_{CC}^2-p_{CD} (p_{DC}+p_{DD}) p_{CC}+p_{DD} (p_{DC}+p_{DD}+1) p_{CC}\\&-p_{CD}+(p_{CD}-p_{DD}) (p_{DD} p_{DC}+p_{DC}+p_{DD})+1\bigr)\Biggr)
\end{cases}
    \end{equation}
    where 
    \begin{equation}
    \begin{split}
    A:=& (p_{CD}-p_{DC}-1) \Bigl((-2 p_{CC}+p_{CD}+p_{DC}+1) p_{DD}^2+2 \left(p_{CC}^2-p_{CD} p_{DC}-1\right) p_{DD}-\\&(p_{CC}-1) (-2 p_{DC} p_{CD}+p_{CD}+p_{DC}+p_{CC} (p_{CD}+p_{DC}-1)-1)\Bigr)^2
    \end{split}
        \end{equation}
\end{footnotesize}
\begin{footnotesize}
    \begin{equation}
        \label{eq: explicit form counting strategy}
        \begin{split}
            dfk
        \end{split}
    \end{equation}
\end{footnotesize}
Explicit equations for the anti-symmetric part of counting strategy  dynamics
    \begin{footnotesize}
\begin{equation}
    \label{eq:anti-symmetric part after multiplication counting}
    \begin{cases}
\dot{q}_2^a  =& -\frac{1}{2} q_0 (2 q_1 (q_1-q_0)+q_0) \bigl(-2 q_0 \left(-q_2^2+q_1^2+1\right)+q_0^2 (-2 q_2+2 q_1+1)+\\ &(q_2-1) (2 q_1 (-q_2+q_1-1)+q_2+1)\bigr),\\
\dot{q}_1^a =& \frac{1}{2} (q_2-1) q_0 (q_2-q_0+1) \bigl(-2 q_0 \left(-q_2^2+q_1^2+1\right)+q_0^2 (-2 q_2+2 q_1+1)+\\&(q_2-1) (2 q_1 (-q_2+q_1-1)+q_2+1)\bigr),\\ 
\dot{q}_0^a =& \frac{1}{2} (q_2-1) (2 q_1 (-q_2+q_1-1)+q_2+1) \bigl(-2 q_0 \left(-q_2^2+q_1^2+1\right)+\\&q_0^2 (-2 q_2+2 q_1+1)+(q_2-1) (2 q_1 (-q_2+q_1-1)+q_2+1)\bigr)
\end{cases}
\end{equation}
\end{footnotesize}
\

\end{document}